\documentclass[reqno,11pt]{amsart}
\usepackage{amsmath,amssymb}
\usepackage{graphicx,color}
\usepackage[notcite,notref]{showkeys}

\newtheorem{theorem}{Theorem}[section]
\newtheorem{proposition}[theorem]{Proposition}
\newtheorem{definition}[theorem]{Definition}
\newtheorem{lemma}[theorem]{Lemma}
\newtheorem{corollary}[theorem]{Corollary}
\newtheorem{remark}[theorem]{Remark}

\numberwithin{equation}{section}
\numberwithin{theorem}{section}

\newcommand{\mc}[1]{{\mathcal #1}}

\newcommand{\bb}[1]{{\mathbb #1}}
\newcommand{\eps}{\varepsilon}
\newcommand{\res}{\mathop{\hbox{\vrule height 7pt width .5pt depth
               0pt\vrule height .5pt width 6pt depth 0pt}}\nolimits}
\newcommand{\upbar}[1]{\,\overline{\! #1}}

\newcommand{\id}{{1 \mskip -5mu {\rm I}}}

\title[Quasi-potentials of entropies for conservation
laws]{Quasi-potentials of the entropy functionals for scalar
  conservation laws}

\author [G.\ Bellettini]{Giovanni Bellettini}
\address{Giovanni
  Bellettini, Dipartimento di Matematica, Universit\`a di Roma `Tor
  Vergata', Via della Ricerca Scientifica 00133 Roma, Italy}
\email{Giovanni.Bellettini@lnf.infn.it}

\author [F.\ Caselli]{Federica Caselli}
\address{Federica Caselli,
  Dipartimento di Ingegneria Civile, Universit\`a di Roma `Tor Vergata',
  Via del Politecnico 1 00100 Roma, Italy}
\email{caselli@ing.uniroma2.it}

\author [M.\ Mariani]{Mauro Mariani}
\address{Mauro Mariani, CEREMADE, UMR-CNRS 7534, Universit\'e de
Paris-Dauphine,
Place du Marechal de Lattre de Tassigny, F-75775 Paris Cedex 16}
\email{mariani@ceremade.dauphine.fr}


\begin{document}
\begin{abstract}
  We investigate the quasi-potential problem for the entropy cost
  functionals of non-entropic solutions to scalar conservation laws
  with smooth fluxes, as defined in \cite{J00,V04,BBMN08,M08}. We
  prove that the quasi-potentials coincide with the integral of a
  suitable Einstein entropy.
\end{abstract}
\maketitle

\section{Introduction}
\label{s:1}
For a real function $f$, consider the scalar conservation law in the
unknown $u \equiv u(t,x)$
\begin{equation}
\label{e:1.1}
  u_t + f(u)_x =0
\end{equation}
where $t \in [0,T]$ for some $T>0$, $x \in \bb T$ (the one-dimensional
torus), and subscripts denote partial derivatives. Equation
\eqref{e:1.1} does not admit in general classical solutions for the
associated Cauchy problem, even if the initial datum is smooth. On the
other hand, if $f$ is non-linear, there exist in general infinitely
many weak solutions. An admissibility condition, the so-called
entropic condition, is then required to recover uniqueness for the
Cauchy problem in the weak sense \cite{D05}. The unique solution
satisfying such a condition is called the \emph{Kruzkhov solution}.

A classical result \cite[Chap. 6.3]{D05} states that the Kruzkhov
solution can be obtained as limit for $\eps \downarrow 0$ of the
solution $u^\eps$ to the Cauchy problem associated with the equation
\begin{equation}
\label{e:1.2}
  u_t+f(u)_x= \frac{\eps}2 \big( D(u)u_x \big)_x
\end{equation}
provided that the initial data also converge. Here the \emph{diffusion
  coefficient} $D$ is a uniformly positive smooth function, and we
remark that convergence takes place in the strong $L_p([0,T]\times \bb
T)$ topology. The Kruzkhov solution to \eqref{e:1.1} has also been
proved to be the hydrodynamical limit of the empirical density of
stochastic particles systems under hyperbolic scaling, when the number
of particles diverges to infinity \cite[Chap.\ 8]{KL99}. These results
legitimize the Kruzkhov solution as the physically relevant solution
to \eqref{e:1.1}, and the entropic condition as the appropriate
selection rule between the infinitely many weak solutions to
\eqref{e:1.1}. 

Provided the flux $f$ and the diffusion coefficient $D$ are chosen
appropriately (depending on the particles system considered), one may
say that \eqref{e:1.2} is a continuous version for the evolution of
the empirical density of particles system, in which the small
stochastic effects are neglected (or averaged) and $\eps$ is the
inverse number of particles. The convergence of both \eqref{e:1.2} and
the empirical measure of the density of particles to the same solution
of \eqref{e:1.1} confirms somehow that this approximation is reliable.

In \cite{J00,V04}, the long standing problem of providing a large
deviations principle for the empirical measure of the density of
stochastic particles systems under hyperbolic scaling is addressed. In
particular, the \emph{totally asymmetric simple exclusion process} is
investigated (which in particular corresponds to $f(u)=u(1-u)$ in the
hydrodynamical limit equation \eqref{e:1.1}), and the large deviations
result partially established. Roughly speaking, when the number of
particles $N$ diverges to infinity, the asymptotic probability of
finding the density of particles in a small neighborhood of a path
$u:[0,T]\times \bb T \to \bb R$ is $e^{-N\, H^{JV}(v)}$, where
$H^{JV}$ is a suitable large deviations rate functional (see
Section~\ref{s:2}).

A continuous mesoscopic mean field counterpart of this large
deviations result is provided in \cite{BBMN08,M08}. In \cite{M08} a
large deviations principle for a stochastic perturbation to
\eqref{e:1.2} (driven by a \emph{fluctuation coefficient} $\sigma$) is
investigated in the limit of jointly vanishing stochastic noise and
(deterministic) diffusion. In \cite{BBMN08} a purely variational
problem is addressed, namely the investigation of the $\Gamma$-limit
of a family of functionals $H_\eps$ associated with \eqref{e:1.2} (see
Section~\ref{s:2}). The candidate large deviations functional $H$
introduced in \cite{M08} and the candidate $\Gamma$-limit introduced
in \cite{BBMN08} coincide, and in the case $f(u)=u(1-u)$ they are
expected to coincide with the functional $H^{JV}$ introduced in
\cite{J00,V04} (the equality can be proved on functions of bounded
variations, but it is missing in the general case). The functional $H$
thus provides a generalization of the functional $H^{JV}$, for
arbitrary fluxes $f$ (in particular, not necessarily convex or
concave), diffusion coefficients $D$ and fluctuation coefficients
$\sigma$. The functionals $H_\eps$, $H$ and $H^{JV}$ are nonnegative;
$H_\eps$ vanishes only on solutions to \eqref{e:1.2}, so that $H_\eps$
can be interpreted as the cost of violating the flow \eqref{e:1.2}. On
the other hand, $H$ and $H^{JV}$ are $+\infty$ off the set of weak
solutions to \eqref{e:1.1}, they vanish only on Kruzkhov solution
to \eqref{e:1.1}, and thus they can be interpreted as the cost of
violating the entropic condition for the flux
\eqref{e:1.1}. Section~\ref{s:2} of the paper is devoted to the
precise definition of the functionals $H_\eps$, $H$ and $H^{JV}$.

Redirecting the reader to Section~\ref{s:3} for a more detailed
discussion, here we briefly recall a general definition of the
quasi-potential associated with a family of functionals. Suppose we
are given a topological space $U$, and for each $T>0$ a set $\mc X_T
\subset C\big([0,T];U\big)$ and a functional $I_T:\mc X_T \to
[0,+\infty]$. For the sake of simplicity, let us also fix a point $m
\in U$. The quasi-potential $V:U \to [0,+\infty]$ associated with
$\{I_T\}$ is then defined as $V(u):= \inf_{T>0} \inf_{w} I_T(w)$,
where the infimum is carried over all the $w \in \mc X_T$ such that
$w(0)\equiv m$ and $w(T)=u$. A natural choice for the reference point
$m$ should be an attractive point for the minima of the functionals
$I_T$ (see e.g.\ Theorem~\ref{t:conve} and Theorem~\ref{t:infconv} for
the case of $H_\eps$, $H$ and $H^{JV}$). Indeed, in such a case, the
investigation of the quasi-potential is a classical subject both in
dynamical optimal control theory and in large deviations theory, as it
quantifies ``the cheapest cost'' to move from the stable point $m$ to
a general one $u$. Moreover, from the optimal control theory point of
view, the quasi-potential describes the long time limit of the
functionals $I_T$, see e.g.\ \cite{BP}. Furthermore, there is a broad
family of stochastic processes for which the quasi-potential is
expected to be the large deviations rate functional of their invariant
measures, provided $I_T$ is the large deviations rate functional of
the laws of such processes up to time $T>0$ (see e.g.\
\cite[Theorem~4.4.1]{FW84} for the classical finite-dimensional case,
and \cite{BFG} for a more general discussion and applications to
particles systems). Moreover, see \cite[Chap.\ 4]{FW84}, the
quasi-potential of the large deviations rate functionals provides a
valuable tool to investigate long time behavior of the processes
(e.g.\ average time to be waited for the process to leave an
attractive point, and the path to follow when the process performs
such a deviation). Finally, in the context of non-equilibrium
statistical mechanics in which the functionals $H_\eps$, $H$ and
$H^{JV}$ have been introduced, the quasi-potential has been proposed
as a dynamical definition of the free energy functional for systems
out of equilibrium \cite{BDGJL}.

Since $H_\eps$ is a functional associated with a
control problem (see \cite{BBMN08}) and it can be also retrieved as
large deviations rate functional of some particles systems (e.g.\
weakly asymmetric particle systems, see \cite{KL99}) and stochastic
PDEs (see \cite{M08b}), the quasi-potential problem is relevant for
such a functional. Similarly, $H$ and $H^{JV}$ are the (candidate)
large deviations rate functionals for both particles systems processes
and stochastic PDEs, see \cite{J00,V04,M08}.

Given a bounded measurable map, $u_i:\bb T \to \bb R$, it is well
known that the (entropic) solutions to the Cauchy problems for
\eqref{e:1.1} and \eqref{e:1.2} with initial datum $u_i$ will converge
to the constant $m = \int_{\bb T} dx\,u_i(x)$, namely constant profiles
are attractive points for the zeros of the functionals $H_\eps$, $H$
and $H^{JV}$. Given $m\in \bb R$, and two positive, smooth maps on
$\bb T$, interpreted as the diffusion coefficient $D$ and fluctuation
coefficient $\sigma$, the \emph{Einstein entropy} is defined as the
unique nonnegative function $h_m$ on $\bb R$ such that $h_m(m)=0$ and
$\sigma h_m''=D$. In this paper, we establish an explicit formula for
the quasi-potential problem associated with the functionals $H_\eps$,
$H$ and $H^{JV}$ (which of course will depend on a time parameter $T$)
with reference point the constant maps on the torus, proving that
these three quasi-potential functionals coincide and are equal to the
integral of the Einstein entropy. More precisely, given $u_f \in
L_\infty(\bb T)$, the quasi-potential $V(m,u_f)$ of $H_\eps$, $H$ and
$H^{JV}$ with reference constant $m\in \bb R$ is equal to $\int_{\bb
  T}dx\,h_m(u_f(x))$ if $\int_{\bb T}dx\, u_f(x) = m$ and it is
$+\infty$ otherwise (see Theorem~\ref{t:quasipot}).

As remarked above, both the large deviations results in
\cite{J00,V04,M08} and the $\Gamma$-limit results in \cite{BBMN08} are
incomplete, due to little knowledge of structure theorems and
regularity results for weak solutions to conservation laws
\eqref{e:1.1}. These results on the quasi-potential give therefore an
additional heuristic argument supporting the actual identification of
$H$ as the $\Gamma$-limit of $H_\eps$. Similarly, since it is easily
seen that the large deviations rate functional (in the hydrodynamical
limit) of the invariant measures of the totally asymmetric simple
exclusion process is also given by the integral of the Einstein
entropy, these results also support the conjecture that $H$ and
$H^{JV}$ may coincide at least in the case $f(u)=u(1-u)$ and that they
are in fact the large deviations rate functional of the totally
asymmetric simple exclusion process. Finally, we remark that the
integral of the Einstein entropy is expected to rule the long time
behavior of well-behaving physical systems, and the result provided
in this paper thus also supports the universality of the Jensen and
Varadhan functional $H^{JV}$ (or in general of $H$) as a relevant
universal \emph{entropy functional} for asymmetric conservative,
closed physical systems.

\section{Preliminaries}
\label{s:2}
Our analysis will be restricted to equibounded ``densities'' $u$, and
for the sake of simplicity we let $u$ take values in $[-1,1]$. Let $U$
denote the compact Polish space of measurable functions $u:\bb T \to
[-1,1]$, equipped with the $H^{-1}(\bb T)$ metric
\begin{equation*}
d_U:=\sup \big\{
       \langle u,\varphi\rangle,\,\varphi \in 
       C^\infty_{\mathrm{c}}\big( \bb T\big),\, 
           \langle \varphi_x, \varphi_x\rangle + \langle \varphi ,\varphi
\rangle=1 
            \big\}
\end{equation*}
where $\langle \cdot,\cdot\rangle$ denotes the scalar product in
$L_2(\bb T)$. Given $T>0$, let $\mc X_T$ be the Polish space
$C\big([0,T]; U\big)$ endowed with the metric
\begin{equation*}
  d_{\mc X_{T}}(u,v):= \sup_{t\in[0,T]} d_U\big( u(t),v(t) \big) +
\|u-v\|_{L_1([0,T]\times \bb T)}
\end{equation*}

Hereafter we assume $f$ a Lipschitz function on $[-1,1]$. Moreover we
let $D,\,\sigma \in C([-1,1])$ with $D$ uniformly positive, and
$\sigma$ strictly positive in $(-1,1)$.

\subsection{The functional $H_{\eps}$}
\label{ss:2.1}
For $\eps>0,\,T>0$, we define $H_{\eps;T}: \mc X_{T} \to [0,+\infty]$
as (hereafter we may drop the explicit dependence on integration
variables inside integrals when no misunderstanding is possible)
\begin{equation}
\label{e:2.3}
H_{\eps;T} (u):=
\begin{cases}
{\displaystyle
\! \sup_{\varphi\in C^\infty_{\mathrm{c}}((0,T)\times \bb T)}
 \!\! \!\! \eps^{-1}\Big[ - \int_{0}^{T}\!\!\!\!
            dt\, \langle  u,\varphi_t\rangle 
        + \langle f(u),\varphi_x\rangle
        - \frac{\eps}2 \langle D(u) u_x,
                           \varphi_x\rangle}
\\
\qquad \qquad {\displaystyle
   - \frac 12 \int_{0}^{T}\!dt\,\langle \sigma(u)\varphi_x ,
                      \varphi_x \rangle \Big]
}
\qquad \text{if $u_x \in
  L^2([0,T]\times \bb T)$}
\\
+\infty\,
\qquad  \text{otherwise}
\end{cases}
\end{equation}
Note that $H_{\eps;T}(u)=0$ iff $u \in \mc X_{T}$ is a
weak solution to \eqref{e:1.2}. $H_{\eps;T}$ is a
lower-semicontinuous and coercive functional on $\mc X_{T}$ (see
\cite[Proposition 3.3, Theorem 2.5]{BBMN08}). Moreover if
$H_{\eps;T}(u)<+\infty$ then $u \in C\big([0,T];L^1(\bb T)
\big)$ (see \cite[Lemma 3.2]{BBMN08}).

\subsection{Entropy-measure solutions}
\label{ss:2.2}
We say that $u \in \mc X_{T}$ is a weak solution to \eqref{e:1.1} iff
for each $\varphi\in C^\infty_{\mathrm{c}}\big((0,T)\times \bb R\big)$
\begin{equation*}
  \int_{0}^{T}\!dt\, \langle u,\varphi_t\rangle
  + \langle f(u),\varphi_x\rangle  = 0
\end{equation*}
A function $\eta \in C^2([-1,1])$ is called an \emph{entropy}, and its
\emph{conjugated entropy flux} $q \in C([-1,1])$ is defined up to an
additive constant by $q(w):=\int^w\!dv\,\eta'(v)f'(v)$. For $u \in \mc
X_{T}$ a weak solution to \eqref{e:1.1}, for $(\eta,q)$ an entropy --
entropy flux pair, the \emph{$\eta$-entropy production} is the
distribution $\wp_{\eta,u}$ acting on
$C^\infty_{\mathrm{c}}\big((0,T)\times\bb R\big)$ as
\begin{equation*}
\wp_{\eta,u}(\varphi):= 
    - \int_{0}^{T}\!dt\,\langle \eta(u) , \varphi_t \rangle
    + \langle  q(u) , \varphi_x \rangle
\end{equation*}
Let $C^{2,\infty}_{\mathrm{c}}\big([-1,1]\times (0,T)\times \mathbb{T}
\big)$ be the set of compactly supported maps
$\vartheta:[-1,1]\times(0,T)\times \bb R \ni (v,t,x) \mapsto
\vartheta(v,t,x) \in \bb R$, that are $C^2$ in the $v$ variable, with
derivatives continuous up to the boundary of $[-1,1]\times(0,T)\times
\bb T$, and $C^\infty$ in the $(t,x)$ variables. For $\vartheta \in
C^{2,\infty}_{\mathrm{c}}\big([-1,1]\times (0,T)\times \bb T \big)$
let $\vartheta'$ and $\vartheta''$ denote its partial derivatives with
respect to the $v$ variable. We say that a function $\vartheta \in
C^{2,\infty}_{\mathrm{c}}\big([-1,1]\times (0,T)\times \bb T \big)$ is
an \emph{entropy sampler}, and its \emph{conjugated entropy flux
  sampler} $Q:[-1,1]\times (0,T)\times \bb T$ is defined up to an
additive function of $(t,x)$ by $Q(w,t,x):=\int^w\vartheta'(v,t,x)
f'(v) dv$. Finally, given a weak solution $u$ to \eqref{e:1.1}, the
\emph{$\vartheta$-sampled entropy production} $P_{\vartheta,u}$ is the
real number
\begin{equation}
\label{e:2.5}
P_{\vartheta,u}
:=-\int_{(0,T)\times \bb T} \! dt\,dx\,
  \Big[\big(\partial_t \vartheta)\big(u(t,x),t,x \big) 
       + \big(\partial_x Q\big)\big(u(t,x),t,x \big)\Big]
\end{equation}
If $\vartheta(v,t,x)=\eta(v) \varphi(t,x)$ for some entropy $\eta$ and
some $\varphi \in C^{\infty}_{\mathrm{c}}\big((0,T)\times \bb T \big)$,
then $P_{\vartheta,u}=\wp_{\eta,u}(\varphi)$.

The next proposition introduces a suitable class of solutions to
\eqref{e:1.1} which will be needed in the sequel. We denote by
$M_T$ the set of finite measures on $(0,T)
\times \bb T$ that we consider equipped with the weak* topology. In
the following, for $\varrho \in M_T$ we
denote by $\varrho^\pm$ the positive and negative part of
$\varrho$.

\begin{proposition} \cite[Proposition 2.3]{BBMN08}, \cite{DOW}.
\label{p:kin}
Let $u \in \mc X_T$ be a weak solution to \eqref{e:1.1}. The
following statements are equivalent:
\begin{itemize}
\item[{\rm (i)}]{for each entropy $\eta$, the $\eta$-entropy
    production $\wp_{\eta,u}$ can be extended to a Radon measure on
    $(0,T)\times \bb T$, namely
    $\|\wp_{\eta,u}\|_{\mathrm{TV}}:=
    \sup \big\{\wp_{\eta,u}(\varphi),\,\varphi \in
               C^\infty_{\mathrm{c}}\big((0,T)\times \bb T \big),\,
               |\varphi|\le 1
         \big\} <+\infty$.}
     \item[{\rm (ii)}]{there exists a bounded measurable map
         $\varrho_u:[-1,1] \ni v \to \varrho_u(v;dt,dx) \in M_T$ such
         that for any entropy sampler $\vartheta$
\begin{equation}
\label{e:2.6}
P_{\vartheta,u} = \int_{[-1,1]\times(0,T) \times \bb T} \! dv\,
          \varrho_u(v;dt,dx)\,\vartheta''(v,t,x)
\end{equation}
}
\end{itemize}
\end{proposition}
We say that a weak solution $u \in \mc X_{T}$ is an
\emph{entropy-measure solution} to \eqref{e:1.1} iff it satisfies the
equivalent conditions of Proposition~\ref{p:kin}. The set of
entropy-measure solutions to \eqref{e:1.1} is denoted by $\mc E_{T}
\subset \mc X_{T}$. In general $\mc E_{T} \nsubseteqq BV([0,T] \times
\mathbb{T})$, the main regularity result for $\mc E_T$ being $\mc
E_{T} \subset C\big([0,T];L^1(\mathbb{T})\big)$, provided $f \in
C^2([-1,1])$ is such that there is no interval in which $f$ is affine
(see \cite[Lemma 5.1]{BBMN08}). A \emph{Kruzkhov solution} to
\eqref{e:1.1} is a weak solution $u \in C\big([0,T];L_1(\bb T)\big)$
such that $\wp_{\eta,u} \le 0$ in distributional sense, for each
convex entropy $\eta$. Since a weak solution $u$ such that
$\wp_{\eta,u} \le 0$ can be shown to be an entropy-measure solution,
the entropic condition is equivalent to $\varrho_u(v;dt,dx) \le 0$ for
a.e. $v \in [-1,1]$.

\subsection{\textbf{$\Gamma$-entropy cost of non-entropic solution}}
\label{ss:2.3}
For $T>0$, we introduce the functional $H_{T}:\mc X_{T} \to
[0,+\infty]$ as
\begin{equation}
  \label{e:2.7}
H_{T}(u):=
 \begin{cases}
  \displaystyle \int \!dv\,\frac{D(v)}{\sigma(v)}\varrho_u^+(v;dt,dx)
           & \quad \text{if $u\in \mc E_{T}$}
 \\
+ \infty  &\quad \text{otherwise}
  \end{cases}
\end{equation}
In \cite[Proposition 2.6]{BBMN08} it is proved that $H_{T}$ is
coercive and lower semicontinuous, and that it vanishes only on
Kruzkhov solutions to \eqref{e:1.1}.

As noted in \cite[Remark 2.7]{BBMN08}, if $u \in \mc X_{T} \cap
BV([0,T]\times \bb {T})$ is a weak solution to \eqref{e:1.1}, then
$u \in \mc E_{T}$. Let $J_u$ be the jump set of $u \in \mc
E_{T}\cap BV([0,T]\times \bb {T})$, $\mc H^1 \res J_u$ the
one-dimensional Hausdorff measure restricted to $J_u$ and, at a point
$(s,y) \in J_u$, let $n=(n^t,n^x)\equiv n(s,y)$ be the normal to $J_u$
and $u^-\equiv u^-(s,y)$ (respectively $u^+\equiv u^+(s,y)$) be the
left (respectively right) trace of $u$ (these are well defined $\mc
H^1 \res J_u$ a.e., since $n^x$ can be chosen uniformly positive, see
\cite[Remark 2.7]{BBMN08}). Then
\begin{equation}
\label{e:2.8}
H_{T}(u) = \int_{J_u} \!d \mc H^1\,|n^x| 
   \int \!dv\,\frac{D(v)}{\sigma(v)} \frac{\rho^+(v,u^+,u^-)}{|u^+-u^-|}
\end{equation}
where
\begin{eqnarray}
\label{e:2.8.5}
\nonumber
 \rho(v,u^+,u^-)& := & \big[f(u^-)(u^+ -v) + f(u^+)(v-u^-)
\\
& & \phantom{ \big[}
                          -f(v)(u^+-u^-) \big]
    \id_{[u^- \wedge u^+,u^- \lor u^+]}(v)
\end{eqnarray}
and $\rho^+$ denotes the positive part of $\rho$.

In \cite{BBMN08} a suitable set $\mc S_{T} \subset \mc
E_{T}$ of \emph{entropy-splittable} solutions to \eqref{e:1.1}
is also introduced, and the next result is proved.
\begin{theorem}\cite[Theorem 2.5]{BBMN08}.
For each $T>0$, the following statements hold.
 \label{t:ecne}
\begin{itemize}
\item[{\rm (i)}] {The sequence of functionals $\{H_{\varepsilon;T}\}$
    satisfies the $\Gamma$-liminf inequality
    $\Gamma$\textrm{-}$\varliminf_\varepsilon H_{\varepsilon;T} \ge
    H_T$ on $\mc X_T$.}

\item[{\rm (ii)}] {Assume that there is no interval where $f$ is
  affine. Then the sequence of functionals $\{H_{\varepsilon;T}\}$ is
  equicoercive on $\mc X_T$.}

\item[{\rm (iii)}] {Assume furthermore that $f \in C^2([-1,1])$, and
  $D,\sigma \in C^{\alpha}([-1,1])$ for some $\alpha>1/2$. Define
\begin{equation*}
  \upbar{H}_T(u) := 
            \inf \big\{\varliminf_n H_T(u_n),\,
               \{u_n\} \subset \mc S_T\,:\:
                 u_n \to u \text{ in $\mc X_T$} \big\}
\end{equation*} 
Then the sequence of functionals $\{H_{\varepsilon;T}\}$ satisfies the
$\Gamma$-limsup inequality $\Gamma$\textrm{-}$\varlimsup_\varepsilon
H_{\varepsilon;T} \le \upbar{H}_T$ on $\mc X_T$.  }

\end{itemize}
\end{theorem}
Note that $\Gamma$-limsup inequality is not complete, as it is not
known that $\upbar{H}_T=H_T$.

\subsection{The Jensen and Varadhan functional}
Suppose that $\sigma$ is such that there exists $h\in
C^2([-1,1])$ such that $\sigma h'' = D$, and let $g$ be such that
$g'=h'f'$. For $T>0$, we also introduce the Jensen and Varadhan
functional $H^{JV}_{T}:\mc X_{T} \to [0,+\infty]$ as
\begin{equation}
H^{JV}_{T}(u):=
 \begin{cases}
\displaystyle 
\! \sup_{\varphi \in C^\infty([0,T]\times \bb T; [0,1])} \!\!
\Big\{
\int_{\bb T}\!\!dx\,
   [h(u(T,x))\varphi(T,x)-h(u(0,x))\varphi(0,x)] &
   \\
\displaystyle 
\phantom{ \displaystyle \! \sup_{\varphi
   \in C^\infty([0,T]\times \bb T; [0,1])} \!\!
\Big\{
} 
- \int_{0}^{T}\!dt\,\langle h(u) , \varphi_t
   \rangle + \langle g(u) , \varphi_x \rangle \Big\} 
\\ \qquad \qquad \qquad \qquad\qquad
\text{if $u$ is
     a weak solution to \eqref{e:1.1}}
   \\
   + \infty \qquad \qquad \qquad \qquad \text{otherwise}
  \end{cases}
\end{equation}
Note that the definition of $H^{JV}_T$ does not depend on the choice
of $h$, provided it satisfies $\sigma h'' = D$. This functional has
been introduced in \cite{J00} (in the case $D\equiv 1$ and
$f(u)=\sigma(u)=u(1-u)$). In \cite{BBMN08} it is proved that $H^{JV}_T
\le H_T$, that $H^{JV}_T(u) = H_T(u)$ if $f$ is convex or concave and
$u$ has bounded variation, and that $H^{JV}_T < H_T$ if $f$ is neither
convex or concave.

\section{Quasi-potentials}
\label{s:3}
We want to study the \emph{quasi-potentials}
$V_{\eps},\,V,\,V^{JV}:[-1,1]\times U \to [0,+\infty]$ associated
respectively with $H_{\eps;T}$, $H_{T}$ and $H^{JV}_{T}$, and defined
as
\begin{equation}
\label{e:3.1}
V_\eps(m,u_f) := \inf \big\{ H_{\eps;T}(u),\,T> 0,\,
             u\in \mc X_T\,:\:u(0)\equiv m,\,u(T)=u_f\big\}
\end{equation}
\begin{equation}
\label{e:3.2}
V(m,u_f) := \inf \big\{ H_T(u),\,T> 0,\,
          u\in \mc X_T\,:\:u(0)\equiv m,\,u(T)=u_f\big\}
\end{equation}
\begin{equation}
\label{e:3.2.1}
V^{JV}(m,u_f) := \inf \big\{ H^{JV}_T(u),\,T> 0,\,
          u\in \mc X_T\,:\:u(0)\equiv m,\,u(T)=u_f\big\}
\end{equation}
Note that, if $u_f\equiv m$, then
$V_\eps(m,u_f)=V(m,u_f)=V^{JV}(m,u_f)=0$. On the other hand, whenever
$m=+1$ or $m=-1$, then if $u_f \not \equiv m$, necessarily $\int\!dx\,
u_f(x) \neq m$ and thus
$V_\eps(m,u_f)=V(m,u_f)=V^{JV}(m,u_f)=+\infty$, since
$H_{\eps;T}(u)=H_{T}(u)=H^{JV}_{T}(u)=+\infty$ whenever $u\in \mc X_T$
is such that $\int_{\bb T}\!dx\, u(s,x) \neq \int_{\bb T}\!dx\, u(t,x) $ for some
$s,\,t\in [0,T]$. Therefore, in the following we focus on the case $m
\in (-1,1)$.

Our main result is the following. For $m \in (-1,1)$ define the
\emph{Einstein entropy} $h_m\in C\big([-1,1];[0,+\infty]\big)\cap
C^2((-1,1))$ as the unique function such that $\sigma(v) h_m^{\prime
  \prime}(v)=D(v)$ for $v \in (-1,1)$, $h_m(m)=0$,
$h_m^{\prime}(m)=0$, and let
\begin{equation*}
W_m(u_f):=\int_{\bb T}\!dx\,h_m(u_f(x)) \in [0,+\infty]
\end{equation*}
Note that, if $\int_{\bb T}\!dx\,u_f(x)=m$, $W_m(u_f)$ can also be
written by the more explicit but less evocative formula
\begin{equation*}
\int_{\bb T}\!dx 
\int_{m}^{u_f(x)}\!dw\,[u_f(x)-w]\frac{D(w)}{\sigma(w)}
\end{equation*}
\begin{theorem}
\label{t:quasipot}
\begin{itemize}
\item[(i)]{ Assume
 \begin{equation}
 \label{e:3.3}
 \lim_{\alpha \downarrow 0} \alpha^2 
    \big[\frac{1}{\sigma(-1+\alpha)}+\frac{1}{\sigma(1-\alpha)} \big]=0
\end{equation}
Then 
\begin{equation*}
V_\eps(m,u_f)=
\begin{cases}
W_m(u_f) & \text{if $\int_{\bb T}\!dx\,u_f=m$}
\\
+\infty & \text{otherwise}
\end{cases}
\end{equation*}
for any $\eps>0$, for any $m\in (-1,1)$ and $u_f \in U$.
}

\item[(ii)]{Assume $f \in C^2([-1,1])$ is such that there is no
    interval in which $f$ is affine. Assume also that for some
    $\delta_0>0$, the set $\{v \in [-1,1]\,:\:f^{\prime \prime}(v)=0\}
    \cap [m-\delta_0,m+\delta_0]$ is finite. Then
\begin{equation*}
V(m,u_f)=
\begin{cases}
W_m(u_f) & \text{if $\int_{\bb T}\!dx\,u_f=m$}
\\
+\infty & \text{otherwise}
\end{cases}
\end{equation*}
for any  $m\in (-1,1)$ and $u_f \in U$.
}

\item[(iii)]{Assume the same hypotheses of (ii) and furthermore that
    there exists $h \in C^2([-1,1])$ such that $\sigma h'' = D$. Then
\begin{equation*}
V^{JV}(m,u_f)=
\begin{cases}
W_m(u_f) & \text{if $\int_{\bb T}\!dx\,u_f=m$}
\\
+\infty & \text{otherwise}
\end{cases}
\end{equation*}
for any  $m\in (-1,1)$ and $u_f \in U$.
}
\end{itemize}
\end{theorem}
Note that \eqref{e:3.3} is verified if $\sigma$ does not vanish, or
vanishes slower than quadratically in $-1$ and $+1$.


Observe that $H_{\eps;T}$ has a quadratic structure (see
\eqref{e:riesz}), so that the proof of Theorem~\ref{t:quasipot}-(i) is
an infinite-dimensional version of Freidlin-Wentzell theorem
\cite[Theorem~4.3.1]{FW84}. However this is not the case for $H_T$. In
particular, since $H_T(u)=+\infty$ if $u$ is not a (entropy-measure)
solution to \eqref{e:1.1}, the main technical difficulty in the proof
of Theorem~\ref{t:quasipot}-(ii) is to show that one can find a
solution $u$ to \eqref{e:1.1} such that $u$ connects \emph{in finite
  time} a profile $v \in U$ close in $L_\infty(\bb T)$ to the constant
profile $m$, to $m$ itself. We remark that the quasi-potential problem
for $H_T$ is at this time being addressed in \cite{B08} in the case of
Dirichlet boundary conditions. While this setting is quite similar to
ours, the difficulties are completely different. In the boundary
driven case, the entropic evolution connects a non-constant profile to
a constant in finite time, so for $T$ large it is not difficult to
solve the minimization problem \eqref{e:3.2} far from the
boundaries. On the other hand, new challenging difficulties appear in
\eqref{e:3.2} when dealing with weak solutions to \eqref{e:1.1}
featuring discontinuities at the boundary (boundary layers). Of course,
this problem does not appear at all on a torus.

\begin{remark}
\label{r:add}
Let $T_1,\,T_2>0$, and let $u_1 \in \mc X_{T_1}$, $u_2 \in \mc X_{T_2}$.
Define the measurable function $u:[0,T_1+T_2]\times \bb T \to [-1,1]$ by
$u(t,x)=u_1(t,x)$ if $t \in [0,T_1]$, and $u(t,x)=u_2(t-T_1,x)$ if $t \in
(T_1,T_1+T_2]$. Then $u \in \mc X_{T_1+T_2}$ iff $u_1(T_1)=u_2(0)$ and
in such a case
\begin{equation*}
H_{\eps;T_1+T_2}(u)=H_{\eps;T_1}(u_1)+H_{\eps;T_2}(u_2)
\end{equation*}

Furthermore if the hypotheses of Theorem~\ref{t:quasipot}-(ii) hold,
then
\begin{equation*}
H_{T_1+T_2}(u)=H_{T_1}(u_1)+H_{T_2}(u_2)
\end{equation*}
\end{remark}
\begin{proof}
  A change of variables in the definition \eqref{e:2.3} shows that
  $H_{\eps;T_1}(u_1)+H_{\eps;T_2}(u_2)$ can be still written in the
  form \eqref{e:2.3} with $T=T_1+T_2$, where however the supremum is
  carried over all the test functions $\varphi \in
  C^\infty_{\mathrm{c}}\big((0,T_1)\cup(T_1,T_1+T_2)\times \bb T
  \big)$.  However, if $u_1(T_1)=u_2(0)$, the supremum carried over
  such test functions coincides with the supremum carried over
  $C^\infty_{\mathrm{c}}\big((0,T_1+T_2)\times \bb T \big)$. Namely,
  $H_{\eps;T_1}(u_1)+H_{\eps;T_2}(u_2)$ equals the definition of
  $H_{\eps;T_1+T_2}(u)$.

  By \eqref{e:2.7} it follows that $H_{T_1+T_2}(u)=+\infty$ whenever
  $H_{T_1}(u_1)=+\infty$ or $H_{T_2}(u_2)=+\infty$. Assume instead
  $H_{T_1}(u_1),\,H_{T_2}(u_2)<\infty$. Under the assumptions of
  Theorem~\ref{t:quasipot}-(ii), the boundedness of $H_T$ implies strong
  continuity in $L_1(\bb T)$ as remarked below
  Proposition~\ref{p:kin}. Therefore if $u_1(T_1)=u_2(0)$ then $u\in
  C\big([0,T_1+T_2];L_1(\bb T)\big)$ and $u \in \mc E_{T_1+T_2}$. By
  \eqref{e:2.5}, \eqref{e:2.6} and the $L_1(\bb T)$ continuity of
  $u_1$, $u_2$ and $u$, it follows that $\varrho_u(v;\{T_1\} \times
  \bb T)=\varrho_{u_1}(v;\{T_1\} \times \bb T)=\varrho_{u_2}(v;\{0\}
  \times \bb T)=0$ for a.e.\ $v\in [-1,1]$. Thus $\varrho_u(v;dt,dx)=
  \varrho_{u_1}(v;dt,dx)$ in $[0,T_1]\times \bb T$ and
  $\varrho_u(v;dt,dx)= \varrho_{u_2}(v;d(t-T_1),dx)$ in
  $[T_1,T_1+T_2]\times \bb T$, and the equality follows from
  \eqref{e:2.7}.
\end{proof}
Since $H_{\eps;T}(m)=H_T(m)=0$, by Remark~\ref{r:add}, the infima in
\eqref{e:3.1}, \eqref{e:3.2} are attained in the limit $T \to
+\infty$.

\section{Proof of Theorem~\ref{t:quasipot} for $V_\eps$}
\label{s:4}
Given a bounded measurable function $a \ge 0$ on $[0,T]\times \bb
T$ let $\mc D^1_{a;T}$ be the Hilbert space obtained by
identifying and completing the functions $\varphi\in
C^\infty([0,T]\times \bb T)$ with respect to the seminorm
$\Big[\int_{0}^{T}\!dt\, \langle \varphi_x, a\, \varphi_x \rangle
\Big]^{1/2}$. Let $\mc D^{-1}_{a;T}$ be its dual space. The
corresponding norms are denoted respectively by $\|\cdot \|_{\mc
  D^1_{a;T}}$ and $\|\cdot \|_{\mc D^{-1}_{a;T}}$.

\begin{remark}
\label{r:dscal}
Let $a \ge 0$ be a bounded measurable function on $[0,T]\times \bb
T$. Let $F,\,G \in L_2\big([0,T]\times \bb T \big)$ be such that
$F_x,\,(a\,G)_x \in \mc D^{-1}_{a;T}$. Assume that $\int_{\bb T}\!dx\,
G(t,x)=0$ for a.e.\ $t \in [0,T]$. Then
 \begin{equation*}
  \Big(F_x, (a\,G)_x \Big)_{\mc D^{-1}_{a;T}} 
     =\int_{0}^{T}\!dt\, \langle F, G\rangle
 \end{equation*}
 where $\big(\cdot, \cdot \big)_{\mc D^{-1}_{a;T}}$ denotes the
 scalar product in $\mc D^{-1}_{a;T}$.
\end{remark}
By a standard application of the Riesz representation theorem (see
\cite[Lemma 3.1]{BBMN08}), we have that if
$H_{\eps;T}(u)<+\infty$ then
\begin{equation}
\label{e:riesz}
H_{\eps;T}(u) =  \frac{\eps^{-1}}{2} \Big\| u_t + f(u)_x 
         -\frac{\eps}2 \big( D(u) u_x\big)_x
          \Big\|_{\mc D^{-1}_{\sigma(u);T}}^2
\end{equation}
If $\int_{\bb T}\!dx \,u_f(x)\neq m$, then
Theorem~\ref{t:quasipot}-(i) follows from the conservation of the
total mass of functions $u \in \mc X_T$ with $H_{\eps;T}(u)<+\infty$.
On the other hand, if $\int_{\bb T}\!dx \, u_f(x)= m$, the proof of
the theorem is a consequence of the following Lemmata. In fact from
Lemma~\ref{l:ubve} we get $V_\eps(m,u_f)\ge W_m(u_f)$, and from
Lemma~\ref{l:ubve} and Lemma~\ref{l:lbve} we have $V_\eps(m,u_f)\le
W_m(u_f) + \gamma$ for each $\gamma>0$.
\begin{lemma}
\label{l:ubve}
Assume \eqref{e:3.3}, let $T>0$ and $u\in \mc X_T$ be such that
$H_{\eps;T}(u)<+\infty$, $u(0,x)\equiv m$, $u(T,x)=u_f(x)$. Then
$\int_{\bb T}\!dx\, h_m(u_f(x))<+\infty$, $u_t + f(u)_x ,\, \big( D(u)
u_x\big)_x \in \mc D^{-1}_{\sigma(u);T}$ and
\begin{equation*}
H_{\eps;T}(u) = \int_{\bb T}\!dx\, h_m(u_f(x))
+\frac{\eps^{-1}}{2} \Big\| u_t + f(u)_x 
         +\frac{\eps}2 \big( D(u) u_x\big)_x
          \Big\|_{\mc D^{-1}_{\sigma(u);T}}^2
\end{equation*}
\end{lemma}

\begin{lemma}
\label{l:lbve}
For each $\gamma>0$, there exists $T>0$ and $u\in \mc X_T$ such that
$H_{\eps;T}(u)<+\infty$, $u(0,x)=m$, $u(T,x)=u_f(x)$ and
\begin{equation*}
\frac{\eps^{-1}}{2} \Big\| u_t + f(u)_x 
         +\frac{\eps}2 \big( D(u) u_x\big)_x
          \Big\|_{\mc D^{-1}_{\sigma(u);T}}^2 \le \gamma
\end{equation*}
\end{lemma}

\begin{proof}[Proof of Lemma~\ref{l:ubve}]
  We first assume that there exists $\delta>0$ such that for a.e.\
  $(t,x) \in [0,T] \times \bb T$, $-1+\delta \le u(t,x) \le 1-\delta$,
  so that $\sigma(u)$ is uniformly positive. It follows that
  $(D(u)u_x)_x,\,f(u)_x \in \mc D^{-1}_{\sigma(u);T}$ so that, since
  $H_{\eps;T}(u)<+\infty$, by \eqref{e:riesz} we also have $u_t \in
  \mc D^{-1}_{\sigma(u);T}$. In particular there exists $\theta \in
  L_2\big([0,T] \times \bb T \big)$ such that $u_t = \theta_x$
  weakly. Therefore
\begin{eqnarray*}
H_{\eps;T}(u) &  = & 
\frac{\eps^{-1}}{2} \Big\| u_t + f(u)_x 
         -\frac{\eps}2 \big( D(u) u_x\big)_x
          \Big\|_{\mc D^{-1}_{\sigma(u);T}}^2
\\
\nonumber
&  = &
\frac{\eps^{-1}}{2} \Big\| u_t + f(u)_x 
         +\frac{\eps}2 \big( D(u) u_x\big)_x
          \Big\|_{\mc D^{-1}_{\sigma(u);T}}^2
\\ \nonumber
& &
 - \Big( \theta_x+f(u)_x, 
         \big( D(u) u_x\big)_x
          \Big)_{\mc D^{-1}_{\sigma(u);T}}
\\
\nonumber
&  = &
\frac{\eps^{-1}}{2} \Big\| u_t + f(u)_x 
         +\frac{\eps}2 \big( D(u) u_x\big)_x
          \Big\|_{\mc D^{-1}_{\sigma(u);T}}^2
\\ \nonumber
& &
 -\int_0^T  \!dt\, \langle \theta , \frac{D(u)}{\sigma(u)} u_x \rangle
+ \langle f(u) , \frac{D(u)}{\sigma(u)} u_x \rangle
\end{eqnarray*}
where in the last line we used Remark~\ref{r:dscal}, as for each $t\in
[0,T]$
\begin{equation*}
\int_{\bb T}\!dx\, \frac{D(u(t,x))}{\sigma(u(t,x))}u_x(t,x) = 
\int_{\bb T}\!dx\, h_m^\prime(u(t,x))_x =0 
\end{equation*}
Similarly we have $\langle f(u(t)) , \frac{D(u(t))}{\sigma(u(t))}
u_x(t) \rangle=0$ and integrating by parts:
\begin{eqnarray*}
 -\int_0^T  \!dt\, \langle \theta , \frac{D(u)}{\sigma(u)} u_x \rangle
& = & \int_0^T  \!dt\, \langle \theta_x , h_m^\prime(u) \rangle
=\int_0^T  \!dt\, \langle u_t , h_m^\prime(u) \rangle
\\
&=& \int_{\bb T} \!dx\, h_m(u(T,x))-h_m(u(0,x))
\end{eqnarray*}
Lemma~\ref{l:ubve} is therefore established for each $u \in \mc X_T$
bounded away from $-1$ and $+1$. For a general $u \in \mc X_T$ such
that $u(0,\cdot)\equiv m\in (-1,1)$, and $\delta>0$, let us define
\begin{equation*}
u^\delta(t,x)=(1-\delta) u(t,x) + \delta m
\end{equation*}
Provided \eqref{e:3.3} holds, the sequence $\{u^\delta\} \subset \mc
X_T$ converges to $u$ as $\delta \to 0$, and is such that: for
$\delta>0$, $u^\delta$ is bounded away from $-1$ and $+1$;
$u^\delta(0,\cdot)\equiv m$, $\int_{\bb T}\!dx\,h(u^\delta(T,x))\to
\int_{\bb T}\!dx\,h(u(T,x))$; $H_{\eps;T}(u^\delta) \to H_{\eps;T}(u)$;
\begin{equation*}
\Big\| u^\delta_t + f(u^\delta)_x +\frac{\eps}2 \big( D(u^\delta) u^\delta_x\big)_x
          \Big\|_{\mc D^{-1}_{\sigma(u^\delta);T}}^2 \to 
\Big\| u_t + f(u)_x +\frac{\eps}2 \big( D(u) u_x\big)_x
          \Big\|_{\mc D^{-1}_{\sigma(u);T}}^2
\end{equation*}
Therefore, since Lemma~\ref{l:ubve} holds for $u^\delta$ for each
$\delta>0$, it also holds for $u$.
\end{proof}

The following result is well known \cite{E97}.
\begin{theorem}
\label{t:conve}
Let $u_f \in U$ and let $v:[0,\infty)\times \bb T \to \bb R$ be the
solution to \eqref{e:1.2} with initial datum $u_f$. Then $\lim_{t\to
  \infty} \|v(t)-m\|_{L_\infty ([0,T]\times \bb T)} =0$ where
$m=\int_{\bb T}\!dx\,u_f(x)$.
\end{theorem}

\begin{proof}[Proof of Lemma~\ref{l:lbve}]
  Let $v:[0,\infty)\times \bb T \to \bb R$ be the solution to
  \eqref{e:1.2} with initial datum $v(0,x)=u_f(-x)$, and for
  $T_1,\,T_2>0$ let $u \in \mc X_{T_1+T_2}$ be defined as
\begin{equation*}
u(t,x)=
\begin{cases}
(1-\frac t{T_1})m + \frac t{T_1} v(T_2,-x)  & \text{for $t\in [0,T_1]$}
\\
v(T_1+T_2-t,-x) & \text{for $t\in [T_1,T_1+T_2]$}
\end{cases}
\end{equation*}
Since $u$ satisfies $u_t+f(u)_x+\frac{\eps}2 (D(u)u_x)_x=0$ for $t\in
[T_1,T_1+T_2]$, we have by Remark~\ref{r:add}
\begin{eqnarray}
\label{e:4.2}
&& \nonumber
\frac{\eps^{-1}}{2} \Big\| u_t + f(u)_x 
         +\frac{\eps}2 \big( D(u) u_x\big)_x
          \Big\|_{\mc D^{-1}_{\sigma(u);T_1+T_2}}^2 
\\
&& \nonumber
\qquad = \frac{\eps^{-1}}{2} \Big\| u_t + f(u)_x 
         +\frac{\eps}2 \big( D(u) u_x\big)_x
          \Big\|_{\mc D^{-1}_{\sigma(u);T_1}}^2 
\\
&& 
\nonumber
\qquad \le  \frac{3 \eps^{-1}}{2} \Big[\big\| u_t  
          \big\|_{\mc D^{-1}_{\sigma(u);T_1}}^2 
+ \big\|f(u)_x  \big\|_{\mc D^{-1}_{\sigma(u);T_1}}^2 
\\ & & \qquad
\phantom{\le  \frac{3 \eps^{-1}}{2} \Big[}
         +\big\|\frac{\eps}2 \big( D(u) u_x\big)_x
          \big\|_{\mc D^{-1}_{\sigma(u);T_1}}^2\Big] 
\end{eqnarray}
Let now $\delta>0$ (to be chosen below) be small enough to have
$-1<m-\delta<m+\delta<1$, and define
\begin{eqnarray*}
C_{\sigma,\delta} &:=& \max_{v \in [m-\delta,m+\delta]}
                      \frac{1}{\sigma(v)}<+\infty
\\
c_{f}(t) & := & \int_{\bb T}\!dx\, \sigma(u(t,x))\,
           \int_{\bb T}\!dx\, \frac{f(u(t,x))}{\sigma(u(t,x))}
\\
C_{f,\delta}& := & \max_{v\in [m-\delta,m+\delta]} 
             f(v) - \min_{v\in [m-\delta,m+\delta]} f(v)
\\
C_D& :=&  \max_{v\in [-1,1]} \frac{D(v)^2}{2}
\end{eqnarray*}
Let also $\theta \in L_2([0,T_1]\times \bb T)$ be defined by
\begin{eqnarray*}
&& \theta_x(t,x)=\frac{v(T_2,-x)-m}{T_1}
\\ &&
\int_{\bb T}\! \frac{\theta(t,x)}{\sigma(u(t,x))}=0
\end{eqnarray*}
By Theorem~\ref{t:conve}, there exists $\tau_\delta>0$ such that
$\|v(t)-m\|_{L_\infty(\bb T)} \le \delta$ for each $t\ge
\tau_\delta$. By Remark~\ref{r:dscal} and \eqref{e:4.2}, since
$u_t=\theta_x$ weakly, we have for each $T_2\ge \tau_\delta$
\begin{eqnarray}
\label{e:4.3}
&& \nonumber
\frac{\eps^{-1}}{2} \Big\| u_t + f(u)_x 
         +\frac{\eps}2 \big( D(u) u_x\big)_x
          \Big\|_{\mc D^{-1}_{\sigma(u);T_1+T_2}}^2 
\\
&& \nonumber
 \le  \frac{3 \eps^{-1}}{2} \int_0^{T_1}\!dt\,
\langle \theta, \frac{\theta}{\sigma(u)} \rangle +\langle f(u) - 
                         c_f, \frac{f(u)-c_f}{\sigma(u)} \rangle 
+\langle D(u)u_x,\frac{D(u)u_x}{\sigma(u)} \rangle
\\
&&
\le  \frac{3 \eps^{-1}}{2}   
 C_{\sigma,\delta} \big[\!
\int_0^{T_1}\!\!\!\!\!dt\,\langle \theta, \theta \rangle + T_1\,C_{f,\delta}^2
+T_1\,C_D  \langle v(T_2)_x, v(T_2)_x \rangle \big]
\end{eqnarray}
By standard parabolic estimates we have
\begin{equation*}
\int_0^{+\infty}\!dt\,\langle v_x,v_x \rangle<+\infty
\end{equation*}
In particular there exists a $T_{2,\delta}>\tau_\delta$ such that
$\langle v(T_{2,\delta})_x, v(T_{2,\delta})_x \rangle <\delta$. Note
that, as $\delta \to 0$, $C_{\sigma,\delta}$ stays bounded, while
$C_{f,\delta}$, $\langle \theta, \theta\rangle$ and $\langle
v(T_{2,\delta})_x, v(T_{2,\delta})_x \rangle$ vanish. Therefore the
right hand side of \eqref{e:4.3} can be made arbitrarily small
provided $\delta$ is small enough.
\end{proof}

\section{Proof of Theorem~\ref{t:quasipot} for $V$ and $V^{JV}$}
\label{s:5}
Define the parity operator $P:U \to U$ by $Pu(x)=u(-x)$ and for $T>0$
the time-space parity operator $P^T:\mc X_T \to \mc X_T$ by $P^T
u(t,x) = u(T-t,-x)$. Define the time reversed quasi-potential
$V:U\times [-1,1] \to [0,\infty]$ as
\begin{equation*}
V(u_i,m) := \inf \big\{ H_T(u),\,T> 0,\,
          u\in \mc X_T\,:\:u(0)= u_i,\,u(T)\equiv m \big\}
\end{equation*}

\begin{lemma}
  \label{l:ubv}
Assume $f \in C^2([-1,1])$ is such that there is no
  interval in which $f$ is affine. Let $T>0$, $u_f\in U$ and
  $m=\int_{\bb T}\!dx \, u_f(x)$. Then
\begin{equation*}
V(m,u_f) = V(Pu_f,m)+ W_m(u_f)
\end{equation*}
\end{lemma}
\begin{proof}
  By the assumptions on $f$, as remarked below
  Proposition~\ref{p:kin}, $\mc E_{T} \subset
  C\big([0,T];L^1(\mathbb{T})\big)$. In particular equations
  \eqref{e:2.5}-\eqref{e:2.6} extend to any $\vartheta \in
  C^{2,\infty}\big([-1,1]\times [0,T]\times \bb T \big)$ (now
  $\vartheta(0)$ and $\vartheta(T)$ need not to vanish) as
\begin{eqnarray}
\label{e:5.2}
\nonumber
& & \int_{\bb T} \!dx\, \vartheta\big(u(T,x),T,x \big)-  
\vartheta\big(u(0,x),0,x \big)
\\
\nonumber
& & -\int_{[0,T]\times \bb T} \! dt\,dx\,
\Big[\big(\partial_t \vartheta)\big(u(t,x),t,x \big) 
+ \big(\partial_x Q\big)\big(u(t,x),t,x \big)\Big] 
\\
\quad
& & =\int_{[-1,1]\times [0,T] \times \bb T} \! dv\,
\varrho_u(v;dt,dx)\,\vartheta''(v,t,x)
\end{eqnarray}
Note that for $u \in \mc E_T$ and $v \in [-1,1]$
\begin{equation*}
\varrho_{P^T
  u}(v;dt,dx)= -\varrho_{u}(v; d(T-t),d(-x))
\end{equation*}
Therefore assuming also $u(0)\equiv m$, $u(T)=u_f$, we have for each
$\eta \in C^2([-1,1])$ with $\eta(m)=0$
\begin{eqnarray}
\label{e:5.3}
& &  \nonumber
\int \!dv\,\eta''(v) \varrho_u^+(v;dt,dx) 
- \int \!dv\,  \eta''(v) \varrho_{P^T u}^+(v;dt,dx)
\\ & &  \nonumber \quad
=\int \!dv\,\eta''(v) \varrho_u^+(v;dt,dx) 
- \int \!dv\,
               \eta''(v) \varrho_u^-(v;d(T-t),-dx)
\\ & & \nonumber \quad
=\int \!dv\,\eta''(v)
                \varrho_u^+(v;dt,dx) 
- \int \!dv\,\eta''(v)
                \varrho_u^-(v;dt,dx)
\\ & &  \nonumber \quad
=\int \!dv\,\eta''(v)
                \varrho_u(v;dt,dx) =
\int\!dx\,\eta(u(T,x))-\eta(u(0,x))
\\ & & \quad
= \int \!dx\,\eta(u_f(x))
\end{eqnarray}
where we used \eqref{e:5.2} with $\theta(v,t,x)=\eta(v)$. If $\sigma$
is bounded away from $0$, then \eqref{e:5.3} evaluated for $\eta=h_m$
immediately yields
\begin{equation}
\label{e:5.3.1}
 H_T(u)= H_T(P^Tu) + W_m(u_f)
\end{equation}
If $\sigma$ vanishes at $-1$ or $+1$, then \eqref{e:5.3.1} is obtained
by monotone convergence, when considering in \eqref{e:5.3} a sequence
$\{\eta^n\} \subset C^2([-1,1])$ such that: $\eta^n(m)=0$; $0\le
(\eta^n)'' \le h_m''$; and for all $v\in [-1,1]$, $\eta^n(v) \uparrow
h_m(v)$ and $(\eta^n)''(v) \uparrow h_m''(v)$.

Optimizing in \eqref{e:5.3.1} over $T$ and $u$ we get $V(m,u_f) \ge
V(Pu_f,m)+ W_m(u_f)$. Replacing $u_f$ by $P u_f$ and thus $P u_f$ by
$P(Pu_f)=u_f$, we get the reverse inequality.
\end{proof}

\begin{definition}
  \label{d:piececon}
  We say that $u_i \in U$ is \emph{piecewise constant} iff there is a
  finite partition of $\bb T$ in intervals such that $u_i$ is constant
  on each interval. For $T>0$, we say that $u \in \mc X_T$ is
  \emph{piecewise constant} iff $u \in C\big([0,T];L_1(\bb T)\big)$
  and there exists a finite partition of $[0,T]\times \bb T$ in
  connected sets with Lipschtiz boundary such that $u$ is constant on
  each set of these.
\end{definition}

The following lemma is the main technical difficulty of this paper,
and its proof is postponed at the end of this section.
\begin{lemma}
  \label{l:connect}
  Assume the same hypotheses of Theorem~\ref{t:quasipot}-(ii).  For
  each $\gamma >0$, there exist $T^\gamma,\,\delta^\gamma>0$ such that
  the following holds. For each \emph{piecewise constant} $u_i \in U$
  satisfying $\int_{\bb T}\!dx\,u_i(x)=m$ and
  $\|u_i-m\|_{L_{\infty}(\bb T)}\le \delta^\gamma$, there exists
  $u^\gamma \in \mc X_{T^\gamma}$ such that $u^\gamma(0)=u_i$,
  $u^\gamma(T^\gamma)\equiv m$ and $H_{T^\gamma}(u^\gamma)\le \gamma$.
\end{lemma}

The next corollary relaxes the condition in Lemma~\ref{l:connect}
requiring $u_i$ to be piecewise constant.
\begin{corollary}
  \label{c:connect}
  Assume the same hypotheses of Theorem~\ref{t:quasipot}-(ii).  For
  each $\gamma >0$, there exist $T^\gamma,\,\delta^\gamma>0$ such that
  the following holds. For each $u_i \in U$ satisfying $\int_{\bb
    T}\!dx\,u_i(x)=m$ and $\|u_i-m\|_{L_{\infty}(\bb T)}\le
  \delta^\gamma$, there exists $u^\gamma \in \mc X_{T^\gamma}$ such
  that $u^\gamma(0)=u_i$, $u^\gamma(T^\gamma)\equiv m$ and
  $H_{T^\gamma}(u^\gamma)\le \gamma$.
\end{corollary}
\begin{proof}
  For a fixed $\gamma>0$, let $T^\gamma$ and $\delta^\gamma>0$ be as
  in Lemma~\ref{l:connect}. For $u_i \in U$ such that $|u_i-m|\le
  \delta^\gamma$, where $m=\int_{\bb T}\!dx\,u_i(x)$, let $\{u_i^n\}
  \subset U $ be a sequence of piecewise smooth functions converging
  to $u_i$ in $U$ and satisfying $\int_{\bb T}\!dx\,u_i^n(x)=m$ and
  $\|u_i-m\|_{L_{\infty}(\bb T)} \le \delta^\gamma$. For each $n$, by
  Lemma~\ref{l:connect} there exist $u^{n,\gamma}$ such that
  $u^{n,\gamma}(0)=u_i^n$, $u^{n,\gamma}(T^\gamma)\equiv m$ and
  $H_{T^\gamma}(u^{n,\gamma})\le \gamma$. Therefore, since
  $H_{T^\gamma}$ has compact sublevel sets (see \cite[Proposition
  2.6]{BBMN08}), there is a (not relabeled) subsequence
  $\{u^{n,\gamma}\}$ converging to a $u^\gamma$ in $\mc X_{T^\gamma}$,
  and $H_{T^\gamma}(u^\gamma)\le \gamma$. By the definition of
  convergence in $\mc X_{T^\gamma}$, $u^{n,\gamma}(0)$ and
  $u^{n,\gamma}(T^\gamma)$ converge in $U$ to $u^\gamma(0)$ and
  $u^\gamma(T^\gamma)$ respectively, and thus $u^\gamma(0)=u_i$ and
  $u^\gamma(T^\gamma)\equiv m$.
\end{proof}

We recall a result in \cite[Chap.\ 11.5]{D05}.
\begin{theorem}
\label{t:infconv}
Assume $f \in C^2([-1,1])$, and that there is no interval in which $f$
is affine. Let $u_i \in U$ and let $\bar{u}:[0,\infty)\times \bb T \to
\bb R$ be the Kruzkhov solution to \eqref{e:1.1} with initial datum
$u_i \in U$. Then 
\begin{equation*}
\lim_{t \to \infty}
\|\bar{u}(t)-m\|_{L_\infty([0,T]\times \bb T)} =0
\end{equation*}
where $m=\int_{\bb T}\!dx\,u_i(x)$.
\end{theorem}

\begin{proof}[Proof of Theorem~\ref{t:quasipot}-(ii)]
  Fix $u_i\in U$ and $\gamma>0$. Let $T^\gamma$ and $\delta^\gamma$ be
  as in Corollary~\ref{c:connect}, let $\bar{u}:[0,+\infty) \to U$ be
  the Kruzkhov solution to \eqref{e:1.1} with initial datum $u_i$, and
  let $m=\int_{\bb T}\!dx\,u_i(x)$. By Theorem~\ref{t:infconv}, there
  exists $\tau^\gamma$ such that $|\bar{u}(\tau^\gamma)-m|\le
  \delta^\gamma$. By Corollary~\ref{c:connect} there exists $\tilde{u}
  \in \mc X_{T^\gamma}$ such that $\tilde{u}(0)=\bar{u}(\tau^\gamma)$,
  $\tilde{u}(T^\gamma)\equiv m$ and $H_{T^\gamma}(\tilde{u}) \le
  \gamma$. Define $u \in \mc X_{\tau^\gamma +T^\gamma}$ by
  \begin{equation*}
u(t,x):= 
  \begin{cases}
    \bar{u}(t,x) & \text{if $t\le \tau^\gamma $}
    \\
    \tilde{u}(t-\tau^\gamma,x) & \text{if $\tau^\gamma \le t \le
      \tau^\gamma+T^\gamma$}
  \end{cases}
  \end{equation*}
  Then, by Remark~\ref{r:add},
  $H_{\tau^\gamma+T^\gamma}(u)=H_{\tau^\gamma}(\bar{u})
  +H_{T^\gamma}(\tilde{u}) \le \gamma $. Therefore $V(u_i,m)=0$ and
  the proof is thus complete since Lemma~\ref{l:ubv} holds.
\end{proof}

The remaining of this section is devoted to the proof of
Lemma~\ref{l:connect}.
\begin{remark}
\label{r:closem}
Let $m\in (-1,1)$, assume the same hypotheses of
Theorem~\ref{t:quasipot}-(ii), and let $\delta_0$ be defined
accordingly. Then, taking perhaps a smaller $\delta_0$, one can assume
$[m-\delta_0,m+\delta_0] \subset (-1,1)$ and that one (and only one)
of the following holds:
\begin{itemize}
\item[\textrm{(A)}]{in the interval $[m-\delta_0,m+\delta_0]$, $f$ is
    either strictly convex or strictly concave.}
\item[\textrm{(B)}]{$f$ is either strictly convex in $[m-\delta_0,m]$
    and strictly concave in $[m,m+\delta_0]$, or strictly concave
    in $[m-\delta_0,m]$ and strictly convex in $[m,m+\delta_0]$.}
\end{itemize}
With no loss of generality, we will assume $f$ convex in
$[m-\delta_0,m+\delta_0]$ if case (A) holds, and $f$ concave in
$[m-\delta_0,m]$ and convex in $[m,m+\delta_0]$ if (B)
holds. 
\end{remark}

\begin{remark}
\label{r:piececon}
Let $T>0$ and assume $u \in \mc E_T$ to be piecewise constant
according to Definition~\ref{d:piececon}. Then the jump set of $u$
consists of a finite number of segments in $[0,T]\times \bb T$. In
particular there exist a finite sequence $0=T^0<T^1<\ldots<T^n=T$,
and, for $k=1,\ldots,n$, finite sequences $\{w^k_j\}_{j=1,\ldots,N_k}
\subset \bb [-1,1]$ such that for $t \in (T^{k-1},T^k)$, $u(t)$ is
piecewise constant with jump set consisting of a finite set of points
$\{x^k_j(t)\}_{k=1,\ldots,N_k} \in \bb T$, and the traces of $u(t)$ at
$x^k_j(t)$ are $w^k_j$ (from the right) and $w^{k}_{j-1}$ (from the
left, where we understand $w^k_0\equiv w^{k}_{N_k}$).

In particular, by \eqref{e:2.8} we have that
\begin{equation}
\label{e:piececost}
  H_T(u)= \sum_{k=1}^n (T^k-T^{k-1}) \sum_{j=1}^{N_k} \int\!dv\,
 \frac{D(v)}{\sigma(v)} \frac{\rho^+(v,w^k_j,w^{k-1}_j)}{|w^{k}_j-w^{k-1}_j|}
\end{equation}
\end{remark}
If $u \in \mc E_T$ is piecewise constant, and $u^-$, $u^+$ are the
left and right traces of $u$ at a given point in the jump set of $u$,
we say that the shock between $u^-$ and $u^+$ is \emph{entropic} iff
$\rho(v,u^-,u^+)\le 0$ for almost every $v$, while it is
\emph{anti-entropic} iff $\rho(v,u^-,u^+)\ge 0$ for almost every
$v$. If $f$ is convex or concave, each shock is either entropic or
anti-entropic, but in the general case the sign of $\rho(v,u^-,u^+)$
may depend on $v$.

\begin{lemma}
  \label{l:connect2}
  Let $m\in (-1,1)$, and $\delta_0 \equiv \delta_0(m)>0 $ be as in
  Remark~\ref{r:closem}. Let $u_i \in U$ be piecewise constant and
  such that $\int_{\bb T}\!dx\,u_i(x)=m$ and $\|u_i-m\|_{L_\infty(\bb
    T)} \le \delta_0$. Then for each $\upbar{T},\,\gamma>0$ there
  exists $w \in \mc X_{\upbar{T}}$ piecewise constant such that
  $\|w-m\|_{L_\infty([0,T]\times \bb T)} \le \|u_i-m\|_{L_\infty(\bb
    T)}$, $w(0)=u_i$, and $H_{\upbar{T}}(w)\le \gamma$.
\end{lemma}
The proof of Lemma~\ref{l:connect2} will be divided in three
steps. The main idea is to construct a piecewise smooth weak solution
$w$, by splitting each shock appearing in the initial datum in an
entropic part and an anti-entropic part, the anti-entropic part being
split itself in $M$ small anti-entropic shocks, with $M$ a large
integer, see Figure~\ref{f:flux}. For such a solution to exist, the
points at which the shocks are split have to be carefully chosen. We
are then able to define $w$ up to the first time at which two (or
more) shocks collide. Defining then $w$ recursively, we prove that
there can be only a finite number of times at which the shocks
collide, and thus $w$ is well-defined globally in time. Finally, we
show that $H_{\upbar{T}}(w)$ can be made arbitrarily small by choosing
$M$ large.

\begin{proof}[Proof of Lemma~\ref{l:connect2}]
  As noted in Remark~\ref{r:closem}, we assume $f$ to be strictly
  convex in $[m,m+\delta_0]$. Hereafter we let
  $\delta:=\|u_i-m\|_{L_\infty(\bb T)} \le \delta_0$.

  \textit{Step 1: Evolution of shocks.}  Let $x_1,\ldots,\,x_N \in \bb
  T$ be the points at which the discontinuities of $u_i$ are
  located. With a little abuse of notation, we also denote by $u_i:\bb
  R \to [-1,1]$ and $x_j\in [0,1]\subset \bb R$ the lift of $u_i$ and
  $x_j$ to $\bb R$, and we assume $x_j < x_{j+1}$ for
  $j=1,\ldots,N-1$. For $j=1,\ldots,N$ and $n\in \bb Z$, let
  $x_{j+nN}=x_j+n \in \bb R$, and for $j \in \bb Z$ let $u^-_j$ and
  $u^+_j$ be respectively the right and left traces of $u_i$ at $x_j$.
  Define
\begin{equation*}
U_j:=
\begin{cases}
  \max \{w\in [u^-_j,u^+_j]\,:\: \rho(v,w,u_j^-)\le 0, \forall v\in [-1,1] \}
 & \text{if $u^-_j <u^+_j $ }
\\
  \min \{w\in [u^+_j,u^-_j]\,:\: \rho(v,w,u_j^-)\le 0, \forall v\in [-1,1] \}
 & \text{if $u^+_j <u^-_j $ }
\end{cases}
\end{equation*}
Since $f$ is convex in $[m,m+\delta_0]$, if $u^-_j <u^+_j $ and $U_j
\le v \le v' \le u^+_j $, or if $u^+_j <u^-_j $ and
$u^+_j \le v' \le v \le U_j $ then
\begin{equation*}
  \frac{f(u^-_j)-f(U_j)}{u^-_j-U_j } 
\le 
  \frac{f(U_j)-f(v)}{U_j - v }
\le 
  \frac{f(v)-f(v')}{v - v' }
\end{equation*}
where we understand $\frac{f(w)-f(w)}{w-w}=f'(w)$ for $w\in [-1,1]$.

\begin{figure}[htb!]
\centering
\includegraphics[scale=0.5]{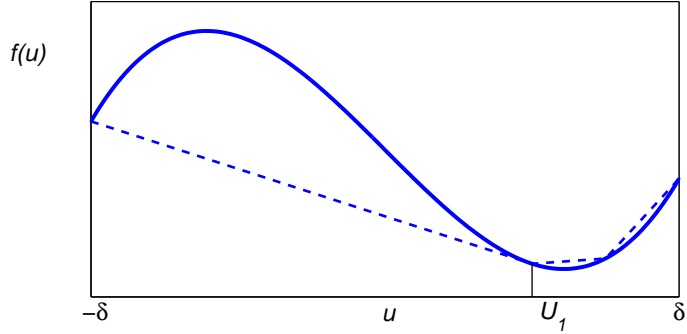}
\caption{In the figure, we have $f(u)=u^3-u$, $m=0$ and
  $M=2$. Consider a discontinuity between the values $u^-_1=-\delta$
  and $u^+_1=\delta$. Then $U_1$ is chosen as the abscissa of the
  point at which a line passing in $(-\delta,f(-\delta))$ is tangent
  to the graph of $f$. The values $U_1+\frac{k}M (u^+_1-U_1)$ for
  $k=1,2$ are the abscissas of the intersections of the dashed lines
  with the graph of $f$.  }
\label{f:flux}
\end{figure}

Therefore, fixed an integer $M\ge 2$, it is possible to define the map
$w_j^{u_i}:[0,+\infty) \times \bb R \to [m-\delta,m+\delta]$ as
\begin{equation}
\label{e:wj}
  w_j^{u_i}(t,x):=
  \begin{cases}
    u^-_j & \text{if $x -x_j \le \frac{f(u_j^-)-f(U_j)}{u_j^- - U_j}t$
    }
    \\
 U_j & \text{if $x -x_j \in [\frac{f(u_j^-)-f(U_j)}{u_j^- - U_j}t,
M \frac{f(U_j)- f( U_j+\frac{u^+_j-U_j}{M})}{U_j-u_j^+}t]$}
  \\
    U_j+\frac{k}M (u^+_j-U_j) & \text{if $x -x_j \in  
    \big[M \frac{f(U_j+\frac{k-1}M
        (u^+_j-U_j))- f( U_j+\frac{k}M (u^+_j-U_j))} 
        {U_j-u_j^+}t,$ }
   \\ & \text{ $M \frac{f(U_j+\frac{k}M
        (u^+_j-U_j))- f( U_j+\frac{k+1}M (u^+_j-U_j))} 
        {U_j-u_j^+}t\big] $} 
   \\ & \qquad \text{ for $k\in \{1,\ldots,\,M-1\}$ }
    \\
    u^+_j & \text{if $x -x_j \ge M \frac{f(U_j+\frac{M-1}M
        (u^+_j-U_j))- f(u^+_j)} {U_j-u_j^+}t$ }
  \end{cases}
\end{equation}
Note that this definition makes sense in the case $U_j=u_j^-$ or
$U_j=u_j^+$. We also let $X_j^-(t):=x_j+\frac{f(u_j^-)-f(U_j)}{u_j^- -
  U_j}t$, $X_j^+(t):=x_j+ \frac{f(U_j+\frac{M-1}M (u^+_j-U_j))-
  f(u^+_j)} {U_j-u_j^+}t$ and
\begin{equation*}
 T(u_i) := \inf\{t\ge 0\,:\:\min_j [X^-_j(t)-X^+_{j-1}(t)]=0 \}
\end{equation*}
We next define $w^{u_i}:[0,T(u_i)] \times \bb R \to
[m-\delta,m+\delta]$ as
\begin{equation*}
  w^{u_i}(t,x) = w_j^{u_i}(t,x) \qquad 
   \text{if $x\in [X_{j}^+(t),X_{j+1}^-(t)]$}
\end{equation*}
$w^{u_i}$ is a weak solution to \eqref{e:1.1} in $[0,T(u_i)] \times
\bb R$, since it is piecewise constant and satisfies the
Rankine-Hugoniot condition along the shocks.  Since it is also
$1$-periodic on $\bb R$, it defines a map $w^{u_i}:[0,T(u_i)]\times
\bb T \to [m-\delta,m+\delta]$ such that $w^{u_i}(0)=u_i$ and $w^{u_i}
\in \mc E_{T(u_i)}$.

\textit{Step 2: There is a finite number of shocks merging.}  We next
define recursively, for $k \ge 1$, $T^k \in [0,\upbar{T}]$ and
$w^k:[T^{k-1},T^k] \times \bb T \to [m-\delta,m+\delta]$ (where
$T^0=0$) by
\begin{eqnarray*}
  & & T^1 := T(u_i)\wedge \upbar{T}
  \\
  & & w^1:=w^{u_i}
  \\
  & & qT^k:=\big[ T^{k-1}+T(w^{k-1}(T^{k-1})) \big] \wedge \upbar{T} 
                  \qquad \text{for $k \ge 2$}
  \\
  & & w^k(t,x):=w^{w^{k-1}(T^{k-1})}(t-T^{k-1},x) \qquad \text{for $k \ge 2$}
\end{eqnarray*}
We want to show that there exists a $K\in \bb N$ such that
$T^K=\upbar{T}$.

By definition, for each $k \ge 1$ and $t\in (T^{k-1},T^k)$, the
discontinuities of $w^k(t)$ are either entropic or non-entropic. On
the other hand, at the times $T^k$ at which two or more shocks
collide, one and only one of the following may happen.
\begin{itemize}
\item[-]{At a point $y \in \bb T$, two or more entropic shocks of
    $w^k$ merge at time $T^k$. Then $w^{k+1}$ has one entropic shock
    in $[T^k,T^{k+1}]$ starting at $y$.}
\item[-]{At a point $y \in \bb T$ one or more entropic shocks of $w^k$
    merge with one or more anti-entropic shocks. Then either $w^{k+1}$
    has one entropic shock starting at $y$, or $w^{k+1}$ has
    $a$ anti-entropic shocks starting at $y$, for some integer $a$,
    $0\le a \le M$.}
\end{itemize}
Note that, at a time $T^k$, one or more of the merging here described
may happen at different points $y \in \bb T$, but at no point there
can be a shock merging involving only anti-entropic shocks. Therefore,
at a given shocks merging: either the number of entropic shocks stays
constant and the number of anti-entropic shocks decreases by at least
one; or the number of entropic shocks decreases by at least one,
and the number of anti-entropic shocks may either decrease, or
increase (by at most $M$). It follows that there can be at most a
finite number of shocks merging, and in particular a finite number of
times at which shocks merge. Recalling that $N$ was the total number
of discontinuity points of $u_i$, and that by definition $w^1$ has at
most $N$ entropic shocks and $N\,M$ anti-entropic shocks, it follows
that for each $k$, $w^k$ has at most $(2\,N-1)M $ anti-entropic
shocks, the remaining shocks being entropic. Therefore the sequence
$\{T^k\}$ has no accumulation points before $\upbar{T}$, and it will hit
$\upbar{T}$ for $k$ large enough.

\begin{figure}[htb!]
\includegraphics[height=80mm,width=60mm]{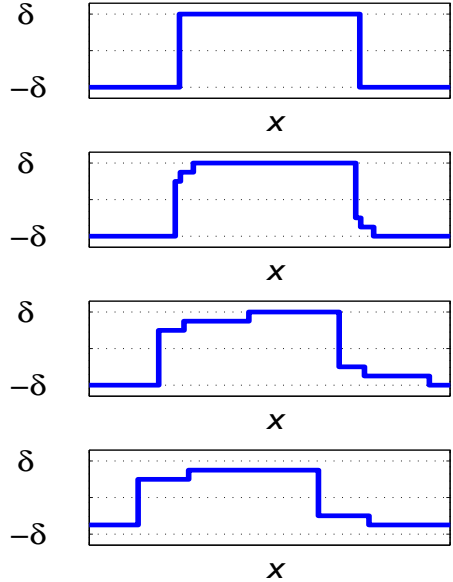}
\caption{In the figure, we have $f(u)=u^3-u$, $m=0$, and the initial
  datum $u_i$ having two jumps between the values $-\delta$ and
  $\delta$. The figure shows $w$ at different times $t\in [0,T(u_i)]$.
}
\label{f:w}
\end{figure}

\textit{Step 3: Computing $H_T$.} We can thus define $w:[0,\upbar{T}]
\times \bb T \to [m-\delta,m+\delta]$ by requiring $w(t,x)=w^k(t,x)$
for $t\in [T^{k-1},T^k]$. $w$ is piecewise constant and it satisfies
the Rankine-Hugoniot condition along the shocks, therefore, since
$w^{k-1}(T^k)=w^k(T^k)$, $w \in \mc E_{\upbar{T}}$. As noted above, in
each time interval $[T^{k-1},T^k]$, $w$ has at most $(2\,N-1)M $
shocks. Moreover, by the definition \eqref{e:wj}, the left and the
right traces of $w$ at an anti-entropic shock differ at most by
$2\delta_0/M$. Therefore we can bound the sum in \eqref{e:piececost}
as
\begin{eqnarray*}
  H_{\upbar{T}}(w)& \le & 
\upbar{T} (2\,N-1)\,M \sup_{\substack{
u^-,u^+\in [m-\delta_0,m+\delta_0]
\\
|u^+-u^-|\le \frac{2\delta_0}{M}} }
\int\!dv\,
 \frac{D(v)}{\sigma(v)} \frac{\rho^+(v,u^+,u^-)}{|u^+-u^-|}
\\ & \le &  
\upbar{T} (2\,N-1)\,M 
 \Big[\max_{v\in [m-\delta_0,m+\delta_0]} \frac{D(v)}{\sigma(v)}\Big]
\\
 & & \quad
\sup_{|u^+-u^-|\le \frac{2\delta_0}{M}}
\Big[ 
\frac{|u^+-u^-|}{2} \big[f(u^+)+f(u^-) \big] 
-\int_{u^-}^{u^+}\!dv\,f(v)
\Big]
\\ & = & 
\upbar{T} (2\,N-1)\,M 
 \Big[\max_{v\in [m-\delta_0,m+\delta_0]} \frac{D(v)}{\sigma(v)}\Big]
\\ & & \quad
\sup_{|u^+-u^-|\le \frac{2\delta_0}{M}}
\Big[ 
\frac{|u^+-u^-|}{2} \big[f(u^+)-f(u^-)-f'(u^-)(u^+-u^-) \big] 
\\ & &   \quad \phantom{
\sup_{|u^+-u^-|\le \frac{2\delta_0}{M}}
\Big[
}
-\int_{u^-}^{u^+}\!dv\,[f(v)-f(u^-)-f'(u^-)(v-u^-)
\Big]
\\ & \le & C\,\upbar{T}\,(2N-1)\,M^{-2}
\end{eqnarray*}
where in the last inequality we used the standard Taylor remainder
estimate and $C$ is a constant depending only on $f,\,
D,\,\sigma$. Namely, $H_{\upbar{T}}(w)$ is arbitrarily small provided
$M$ is chosen sufficiently large.
\end{proof}

In the following, whenever $m+\delta,\,m+\delta' \in [-1,1]$, we
introduce the short hand notation $R(\delta,\delta')$ for the
Rankine--Hugoniot velocity of a shock between the values $m+\delta$
and $m+\delta'$, namely
\begin{equation*}
  R(\delta,\delta'):= \frac{f(m+\delta)-f(m+\delta')}{\delta-\delta'}
\end{equation*}
and we understand $R(\delta,\delta)=f'(m+\delta)$. We also introduce
\begin{eqnarray*}
 C(\delta,\delta') & := &  
\int \!dv\, \frac{\rho(v,m+\delta,m+\delta')}{|\delta-\delta'|}
\\ 
& = & \frac{|\delta-\delta'|}{2} \big[f(m+\delta)+f(m+\delta') \big] 
-\int_{m+\delta'}^{m+\delta}\!dv\,f(v)
\end{eqnarray*}

The following lemma introduces an explicit solution to \eqref{e:1.1},
with initial datum having only two discontinuities and final datum
being constant.
\begin{lemma}
  \label{l:connect1}
  Assume the same hypotheses of Theorem~\ref{t:quasipot}-(ii) and let
  $\gamma>0$. Let $m \in (-1,1)$ and let $\delta_0 =\delta_0(m)$ be
  defined as in Remark~\ref{r:closem}. Then for each $\delta_1 \in
  (0,\delta_0)$ there exists $\bar{\delta}_2\equiv
  \bar{\delta}_2(\delta_1) \in (0,\delta_0)$ such that for each
  $\delta_2 \in (0,\bar{\delta}_2)$ the following holds. For a fixed
  arbitrary $x_0 \in \bb T$ let $u_d\equiv u_d^{\delta_1,\delta_2}\in
  U$ be defined as
\begin{equation*}
  u_d(x):=
  \begin{cases}
    m+\delta_1 & \text{if $|x-x_0|\le \frac{\delta_2}{2(\delta_1+\delta_2)} $}
\\
    m-\delta_2 & \text{otherwise}
  \end{cases}
\end{equation*}
and let
\begin{equation*}
  \tau\equiv \tau^{\delta_1,\delta_2}:=
  \frac{1}{|R(\delta_1,0) -R(0,-\delta_2)|}
\end{equation*}
Then $\tau<\infty$, and there exists $u \in \mc X_\tau$ such that
$u(0)=u_d$, $u(\tau)\equiv m$ and 
\begin{equation}
\label{e:ucost1}
H_\tau(u) \le  \Big[\max_{v\in [m-\delta_0,m+\delta_0]} \frac{D(v)}{\sigma(v)}\Big]
\frac{C(\delta_1,0)^+ + C(0,-\delta_2)^+}{|R(\delta_1,0) -R(0,-\delta_2)|}
\end{equation}
\end{lemma}

\begin{proof}
  Fix $\delta_1 \in (0,\delta_0)$. By the definition of $\delta_0$,
  $R(\delta_1,0)\neq f'(m)$ and assuming $f$ strictly convex in
  $[m,m+\delta_0]$ (see Remark~\ref{r:closem}), we have $R(\delta_1,0)
  > f'(m)$. Recalling the definition of $\rho$ in \eqref{e:2.8.5},
  still by the convexity of $f$ in $[m,m+\delta_0]$, we have
  $\rho(v,m,m+\delta_1)< 0$ for $v \in (m,m+\delta_1)$ and
  $C(\delta_1,0)>0$. In particular there exists $\bar{\delta}_2$ small
  enough such that for each $\delta_2 \in (0,\bar{\delta}_2)$ and each
  $v \in (m-\delta_2,m+\delta_1)$
 \begin{equation}
\label{e:d2small1}
R(\delta_1,0)-R(0,-\delta_2)>0
\end{equation}
\begin{equation}
\label{e:d2small2}
 \rho(v,m-\delta_2,m+\delta_1)< 0
\end{equation}

Let us now fix $\delta_2 \in (0,\bar{\delta}_2)$. By
\eqref{e:d2small1} $\tau^{\delta_1,\delta_2}$ is finite. With no loss
of generality we may assume $x_0=
\frac{\delta_2}{2(\delta_1+\delta_2)}$, as the general case is obtained
by a space translation of the solution $u$ given below by the
quantity $x_0-\frac{\delta_2}{2(\delta_1+\delta_2)}$. Define
  \begin{equation*}
   u(t,x):=
   \begin{cases}
     m & \text{ if $\big|x- [R(\delta_1,0)+R(0,-\delta_2)]\frac{t}2
       \big|$}
\\
{ } &  \text{\phantom{if }
           $\le [R(\delta_1,0)-R(0,-\delta_2)]\frac{t}2 $}
\\
     m+\delta_1 & \text{if $\big|x-
       \frac{\delta_2}{2(\delta_1+\delta_2)}-
       [R(\delta_1,0)+R(\delta_1,-\delta_2)]\frac{t}2 \big|$}
\\
{ } &  \text{\phantom{if } $ \le 
       \frac{\delta_2}{2(\delta_1+\delta_2)}-
     [R(\delta_1,0)-R(\delta_1,-\delta_2)]\frac{t}2$}
\\
    m-\delta_2 & \text{otherwise}
   \end{cases}
 \end{equation*}
 It follows that $u(0)=u_d$ and $u(\tau)\equiv m$. Moreover $u$ is a
 piecewise constant weak solution to \eqref{e:1.1}. For a fixed $t\in
 (0,T)$, $u(t)$ has three discontinuity points, where its value jumps
 from $m$ to $m+\delta_1$, from $m+\delta_1$ to $m-\delta_2$ and from
 $m-\delta_2$ to $m$. In particular $H_\tau(u)$ can be calculated by
 \eqref{e:piececost}. The shock between the values $m+\delta_1$ and
 $m-\delta_2$ is entropic by \eqref{e:d2small2}, and thus it gives no
 contributions to the sum \eqref{e:piececost}. By the convexity
 assumption on $f$ in $[m,m+\delta_0]$, the shock between $m$ and
 $m+\delta_1$ is anti-entropic, namely $\rho(v, m+\delta_1,m) \ge
 0$. Moreover the shock between $m-\delta_2$ and $m$ is either
 entropic (if case (A) in Remark~\ref{r:closem} holds) or
 anti-entropic (if case (B) in Remark~\ref{r:closem} holds). Therefore
 \eqref{e:piececost} yields
 \begin{eqnarray*}
H_{\tau}(u) & = & 
\tau \Big[\int \!dv\,\frac{D(v)}{\sigma(v)} 
\frac{\rho^+(v,m,m-\delta_2)}{\delta_2} +
\int \!dv\,\frac{D(v)}{\sigma(v)} 
\frac{\rho^+(v,m+\delta_1,m)}{\delta_1}   \Big]
\\ & = &
\frac{
\Big[\int \!dv\,\frac{D(v)}{\sigma(v)} 
\frac{\rho(v,m,m-\delta_2)}{\delta_2}\Big]^+ +
\int \!dv\,\frac{D(v)}{\sigma(v)} 
\frac{\rho(v,m+\delta_1,m)}{\delta_1}
}{R(\delta_1,0) -R(0,-\delta_2)}
\\ &\le &
\Big[\max_{v\in [m-\delta_0,m+\delta_0]} \frac{D(v)}{\sigma(v)}\Big]
\frac{
\Big[\int \!dv\,
\frac{\rho(v,m,m-\delta_2)}{\delta_2}\Big]^+ +
\int  \!dv\, 
\frac{\rho(v,m+\delta_1,m)}{\delta_1}
}{R(\delta_1,0) -R(0,-\delta_2)}
\end{eqnarray*}
namely \eqref{e:ucost1}.
\end{proof}

\begin{proof}[Proof of Lemma~\ref{l:connect}]
  Fix $\gamma>0$. Recall the definition of $\delta_0 \equiv
  \delta_0(m)$ in Remark~\ref{r:closem}; as noted in
  Remark~\ref{r:closem} we may assume $f$ to be strictly convex in
  $[m,m+\delta_0]$. We thus have $R(\delta_1,0)>f'(m)$,
  $\rho(v,m+\delta_1,m)\ge 0$ for each $\delta_1 \in
  (0,\delta_0)$. Then by explicit computation
  \begin{equation*}
    \lim_{\delta_1 \downarrow 0} \lim_{\delta_2 \downarrow 0} 
\lim_{\delta \downarrow 0}  \sup_{\delta',\delta'' \in [-\delta,\delta]} 
\frac{|C(\delta_1,\delta')| + |C(\delta'',-\delta_2)|+|C(\delta_1,-\delta_2)|}
{R(\delta_1,0) -f'(m)}=0
  \end{equation*}

  In particular, defining $\bar{\delta}_2(\cdot)$ as in
  Lemma~\ref{l:connect1}, there exist $\delta_1 \equiv
  \delta_1^\gamma\in (0,\delta_0)$, $\delta_2 \equiv
  \delta_2^\gamma\in (0,\bar{\delta}_2(\delta_1))$ and $\delta \equiv
  \delta^\gamma \in (0,\delta_1 \wedge \delta_2)$ such that
\begin{equation}
\label{e:ucost1b}
\Big[\max_{v\in [m-\delta_0,m+\delta_0]} \frac{D(v)}{\sigma(v)}\Big]
 \sup_{\delta' \in [-\delta,\delta]} 
\frac{C(\delta_1,0) + |C(\delta',-\delta_2)|+C(\delta_1,-\delta_2)}
{R(\delta_1,0) -f'(m)}\le \frac{\gamma}{8}
\end{equation}
\begin{eqnarray}
\label{e:ucost2}
\nonumber
\frac{R(\delta_1,0)-f'(m)}{4}
& \le &  \frac{R(\delta_1,0)-R(0,-\delta_2)}{2}
\\
& \le & \inf_{\delta',\delta'' \in [-\delta,\delta]}
R(\delta_1,\delta')-R(\delta'',-\delta_2)
\end{eqnarray}
\begin{equation}
\label{e:ucost2b}
\inf_{\delta',\delta'' \in [-\delta,\delta]}
|R(\delta',-\delta_2)-R(\delta',\delta'')|>0
\end{equation}
\begin{equation}
\label{e:ucost2c}
\inf_{\delta',\delta'' \in [-\delta,\delta]}
|R(\delta_1,\delta'')-R(\delta',\delta'')|>0
\end{equation}
\begin{equation}
  \label{e:ucost3}
\rho(v,m-\delta,m+\delta_1)> 0     
   \quad \text{for $v\in (m-\delta,m+\delta_1)$}
\end{equation}
\begin{equation}
  \label{e:ucost4}
|\rho(v,m+\delta,m-\delta_2)|> 0 
  \quad \text{for $v\in (m-\delta_2,m+\delta)$}
\end{equation}

Let now $u_i\in U $ be an arbitrary piecewise constant profile such
that $\|u_i-m\|_{L_{\infty}(\bb T)}\le \delta$. Fix
\begin{equation*}
\upbar{T}:=  \frac{4}{R(\delta_1,0)-f'(m)}
\end{equation*}
By Lemma~\ref{l:connect2} there exists a piecewise constant map
$w\equiv w^{\upbar{T},\gamma/4} :[0,\upbar{T}]\times \bb T \to
[m-\delta,m+\delta]$ such that $w(0)=u_i$ and $H_{\upbar{T}}(w)\le
\gamma/4$.

Let the Lipschitz map $s_1,\,s_2:[0,+\infty)\to \bb T$ be defined as
the solutions to the Cauchy problems
\begin{equation*}
  \begin{cases}
\dot{s}_1(t)= \frac{f(m+\delta_1)-f(w(t,s_1(t)))}{m+\delta_1-w(t,s_1(t))} 
            \equiv R(\delta_1,w(t,s_1(t))-m)  
\\
s_1(0)=0
  \end{cases}
\end{equation*}

\begin{equation*}
  \begin{cases}
\dot{s}_2(t)= \frac{f(m-\delta_2)-f(w(t,s_2(t)))}{m-\delta_2-w(t,s_2(t))}  
\equiv R(w(t,s_2(t))-m,-\delta_2)   
\\
s_2(0)=0
  \end{cases}
\end{equation*}
D espite the right hand sides are discontinuous, these equations are
well-posed since $w$ is piecewise constant and conditions
\eqref{e:ucost2b}-\eqref{e:ucost2c} hold.

With a little abuse of notation, we also denote by $s_1$ and $s_2$ the
lift of $s_1$ and $s_2$ on $\bb R$. Note that, by \eqref{e:ucost2},
$s_1(t)-s_2(t)$ is increasing in $t$ and letting $T>0$ be the first
time $t$ at which $s_1(t)-s_2(t)=1$, we have still by \eqref{e:ucost2}
\begin{equation}
  \label{e:tsmall}
  T \le \upbar{T}
\end{equation}
We also set $x_0:=s_1(T)\equiv s_2(T)\in \bb T$, and let $u_d\equiv
u_d^{\delta_1,\delta_2}$, $\tau\equiv \tau^{\delta_1,\delta_2}$ be
defined as in Lemma~\ref{l:connect1} (with $\delta_1,\,\delta_2$ and
$x_0$ defined as above in this proof), and let $v \in \mc X_\tau$ be
the solution to \eqref{e:1.1} whose existence is proved in
Lemma~\ref{l:connect1}. We finally let
\begin{equation*}
  u^\gamma(t,x):=
  \begin{cases}
m+\delta_1   & \text{if $t \in [0,T] $ and 
        $x\in A_1(t)$}
\\
m-\delta_2   & \text{if $t \in [0,T] $ and 
        $x\in A_2(t)$}
\\
w(t,x) & \text{if $t \in [0,T] $ and 
        $x\not\in A_1(t)\cup A_2(t)$}
\\
v(t-T,x)  & \text{if $t \in [T,T+\tau]$}    
  \end{cases}
\end{equation*}
where for $t\ge 0$
\begin{eqnarray*}
 & & A_1(t):= \big\{x\in \bb T\,:\:\big|x
- \frac 12 \big[s_1(t)+R(\delta_1,-\delta_2)t \big] \big| 
\\ & & 
\phantom{
 A_1(t):= \big\{x\in \bb T\,:\:
}\quad
\le 
\frac 12 \big[s_1(t)
        - R(\delta_1,-\delta_2)t \big]  \big\}
\\
& &   A_2(t):=  \big\{x\in \bb T\,:\:\big|x
 -  \frac 12 \big[R(\delta_1,-\delta_2)t +s_2(t) \big]\big|
\\ & & 
\phantom{
  A_2(t):=  \big\{x\in \bb T\,:\:
} \quad
\le \frac 12 \big[R(\delta_1,-\delta_2)t -s_2(t)\big]  \big\}
\end{eqnarray*}
Note that $u^\gamma_{|{[0,T]}} \in \mc X_T$ is piecewise constant, and
it is the gluing of solutions to \eqref{e:1.1} satisfying the
Rankine-Hugoniot condition at the borders of $\{(t,x)\in [0,T]\times
\bb T\,:\: x\in A_i(t)\}$ (for $i=1,2$). We thus have
$u^\gamma_{|{[0,T]}} \in \mc E_T$ and $u^\gamma \in \mc E_{T+\tau}$.
\begin{figure}[htb!]
\includegraphics[height=80mm,width=60mm]{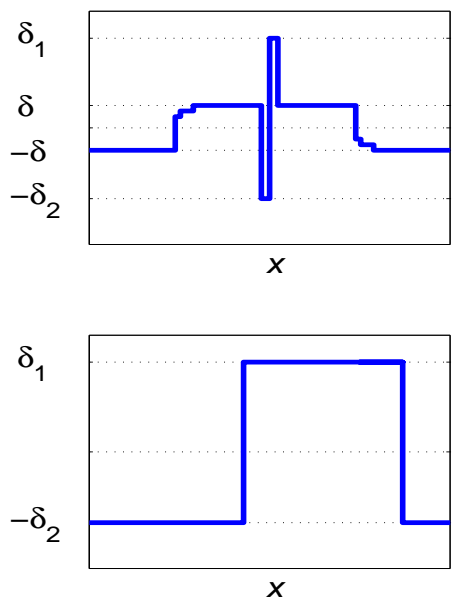}
\caption{In the figure, we have $f(u)=u^3-u$, $m=0$, and the initial
  datum $u_i$ having two jumps between the values $-\delta$ and
  $\delta$. The figure shows $u^\gamma$ at a small time $0<t<T$ and at
  time $T$.  }
\label{f:w53}
\end{figure}
\begin{figure}[htb!]
\includegraphics[height=80mm,width=60mm]{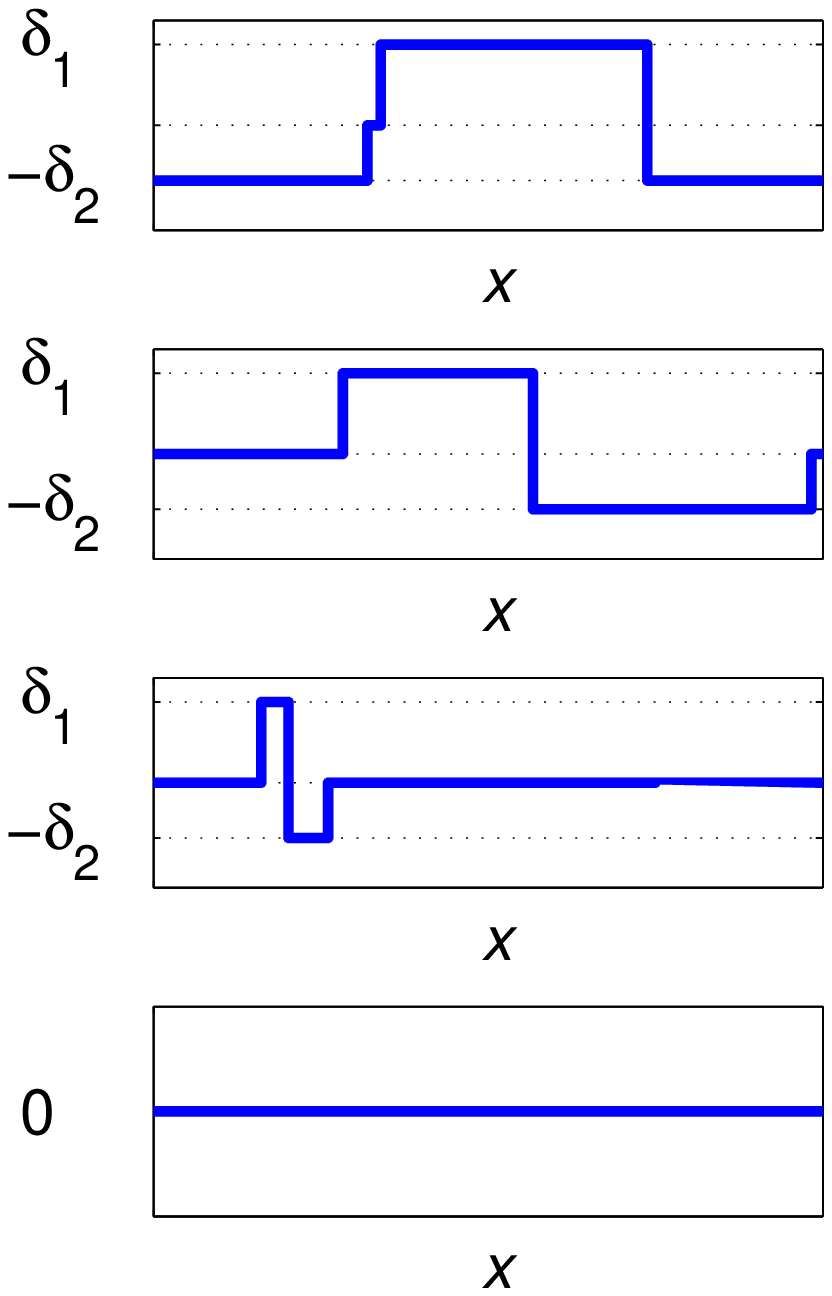}
\caption{In the figure, we have $f(u)=u^3-u$, $m=0$, and the initial
  datum $u_i$ having two jumps between the values $-\delta$
  and $\delta$. The figure shows $u^\gamma$ at different times $t\in
  (T,T+\tau]$.  }
\label{f:w59}
\end{figure}

In order to calculate $H_T(u^\gamma_{|_{[0,T]}})$, we will use
Remark~\ref{r:piececon}. Note that for each $t\in [0,T]$ the set of
discontinuity points of $u^\gamma(t)$ consists of the discontinuity
points of $w(t)$, and the discontinuities at $s_1(t)$, at $s_2(t)$ and
at $R(\delta_1,-\delta_2)t$. Because of assumptions
\eqref{e:ucost2b}-\eqref{e:ucost2c}, there is at most a finite number
of times $t \in [0,T]$ at which $s_1(t)$ and $s_2(t)$ may overlap with
a discontinuity of $w(t,\cdot)$. Note that assumption \eqref{e:ucost3}
implies $\rho(v,w,m+\delta_1)\le 0$, for each $v\in [-1,1]$ and $w \in
[m-\delta,m+\delta]$, so that the shock of $u^\gamma$ at $s_1$ is
entropic and it does not appear in the sum
\eqref{e:piececost}. Conversely $\rho(v,m-\delta_2,m+\delta_1)\ge 0$,
so that the shock along the curve $t \mapsto R(\delta_1,-\delta_2)t$
appears in the sum \eqref{e:piececost}. Finally, by \eqref{e:ucost4},
$\rho(v,w(t,s_2(t),m-\delta_2)$ is either negative or positive for
each $t\in [0,T]$ and $v\in [-1,1]$, depending whether case (A) or (B)
of Remark~\ref{r:closem} holds for $f$. By Remark~\ref{r:add} and
recalling that $v$ satisfies \eqref{e:ucost1}
\begin{eqnarray*}
&  &  H_{T+\tau}(u^\gamma)  = 
  H_\tau(v) + H_T(u^\gamma_{|{[0,T]}})
\le H_\tau(v) + H_{\upbar{T}}(w) 
\\ & & \quad
+\int_0^T\,dt\,\Big[\int
 \!dv\,\frac{D(v)}{\sigma(v)} 
\frac{\rho^+(v,m-\delta_2,w(t,s_2(t))}{\delta_2} +
\\ & & \quad \phantom{
+\int_0^T\,dt\,\Big[
}
\int \!dv\,\frac{D(v)}{\sigma(v)} 
\frac{\rho^+(v,m+\delta_1,m-\delta_2)} {\delta_1+\delta_2}  \Big]
\\ & \le & 
\frac{\gamma}4 
+  \Big[\max_{v\in [m-\delta_0,m+\delta_0]} \frac{D(v)}{\sigma(v)}\Big] 
  \frac{C(\delta_1,0)^+ + C(0,-\delta_2)^+}{|R(\delta_1,0) -R(0,-\delta_2)|}
\\ & & 
+ \int_0^T\,dt\,\Big[\int
 \!dv\,\frac{D(v)}{\sigma(v)} 
\frac{\rho(v,m-\delta_2,w(t,s_2(t))}{\delta_2}\Big]^+
\\ & & 
+ \int \! dv\, \frac{D(v)}{\sigma(v)} 
\frac{\rho(v,m+\delta_1,m-\delta_2)} {\delta_1+\delta_2}
\\ & \le &
\frac{\gamma}4
+  \Big[\max_{v\in [m-\delta_0,m+\delta_0]} \frac{D(v)}{\sigma(v)}\Big] 
\\ & & \quad
\Big[ \frac{C(\delta_1,0)^+ + C(0,-\delta_2)^+}
                          {|R(\delta_1,0) -R(0,-\delta_2)|}
+ T \, C(\delta_1,-\delta_2)
+ T \, \sup_{\delta'\in [-\delta,\delta]} C(\delta',-\delta_2)^+ \Big]
\end{eqnarray*}
By \eqref{e:tsmall} and \eqref{e:ucost2} we thus obtain
\begin{eqnarray*}
& &  H_{T+\tau}(u^\gamma)) \le \frac{\gamma}4 + 
\Big[\max_{v\in [m-\delta_0,m+\delta_0]} \frac{D(v)}{\sigma(v)}\Big]
\\ & & \,
 \frac{2 C(\delta_1,0)^+ + 2 C(0,-\delta_2)^+ + 4 C(\delta_1,-\delta_2)
+ 4 \sup_{\delta'\in [-\delta,\delta]} C(\delta',-\delta_2)^+}
 {R(\delta_1,0) -f'(m)}
\\ & & 
\le  
\frac{\gamma}4 
+ 6 \Big[\max_{v\in [m-\delta_0,m+\delta_0]} \frac{D(v)}{\sigma(v)}\Big] 
\\ & & \qquad \qquad \qquad 
\frac{C(\delta_1,0) + C(\delta_1,-\delta_2)
+ \sup_{\delta'\in [-\delta,\delta]} |C(\delta',-\delta_2)|}
 {R(\delta_1,0) -f'(m)}
\end{eqnarray*}
Therefore $H_{T+\tau}(u^\gamma)) \le \gamma$ by \eqref{e:ucost1b}.
\end{proof}

\begin{proof}[Proof of Theorem~\ref{t:quasipot}-(iii)]
  We assume $\int_{\bb T}dx\, u_f(x)= m$, the proof being trivial
  otherwise. Since $H^{JV}_T \le H_T$ we have $V^{JV} \le W_m(u_f)$ by
  Theorem~\ref{t:quasipot}-(ii). The converse inequality is obtained
  by taking $\varphi \equiv 1$ in the very definition of $H^{JV}$.
\end{proof}

\textit{Acknowledgment:} We are indebted to Lorenzo Bertini and Matteo Novaga
for enlighting discussions.

\end{document}